\documentclass[10pt,twoside,a4paper]{amsart}

\usepackage{amsfonts,amsmath,amsthm}
\usepackage{hyperref}

\newtheorem{thm}{Theorem}[section]

\newtheorem{lem}[thm]{Lemma}
\newtheorem{pro}[thm]{Proposition}

\theoremstyle{remark}
\newtheorem{remark}[thm]{Remark}
\theoremstyle{ass}
\newtheorem{ass}[thm]{Assumption}

\newcommand{\R}{\mathbb{R}}
\newcommand{\Rd}{\mathbb{R}^{d}}
\newcommand{\D}{\nabla_{A}}
\newcommand{\C}{\mathcal{C}^{*}}

\numberwithin{equation}{section}

\begin{document}

\title[Helmholtz equation with magnetic potential]{Limiting absorption principle for the electromagnetic Helmholtz equation with singular potentials}

\author{Miren Zubeldia}

\address{M. Zubeldia: Universidad del Pa\'is Vasco, Departamento de Matem\'aticas, Apartado 644, 48080, Spain}
\email{miren.zubeldia@ehu.es}

\thanks{The author is supported by the Spanish grant FPU AP2007-02659 of the MEC}

\begin{abstract} We study the following Helmholtz equation 
\begin{equation}\label{abstract}
(\nabla +iA(x))^{2} u+ V_{1}(x) u + V_{2}(x) u  + \lambda u  = f(x)\notag
\end{equation}
in $\Rd$ with magnetic and electric potentials that are singular at the origin and decay at infinity. We prove the existence of a unique solution satisfying a suitable Sommerfeld radiation condition, together with some a priori estimates. We use the limiting absorption method and a multiplier technique of Morawetz type.
\end{abstract}

\date{\today}


\maketitle

\section[Introduction]{Introduction}
Let us consider the electromagnetic Schr\"odinger operator
\begin{equation}\notag
L=\sum_{j=1}^{d} (\nabla_{j} + iA_{j})^{2} + V
\end{equation}
in the Hilbert space $L^{2}(\Rd)$, $d\geq 3$. Here $A: \Rd \to \Rd$ is the magnetic vector potential and $V:\Rd \to \R$ is the electric scalar potential. We are interested in studying solutions of the equation
\begin{equation}\label{introISequ}
(L+\lambda) u = f, \quad \quad \lambda >0
\end{equation}
where $f$ is a suitable function on $\Rd$.

The standard covariant form of the electromagnetic Schr\"odinger hamiltonian is
\begin{equation}\notag
L = \D^{2} + V
\end{equation}
with
\begin{equation}\label{1.4}
\D = \nabla + i A.
\end{equation}
The magnetic potential $A$ describes the interaction of a free particle with an external magnetic field. The magnetic field that corresponds to a magnetic potential $A$ is given by the $d \times d$ anti-symmetric matrix defined by
\begin{equation}\label{rotacional}
B=(DA)-(DA)^{t}, \quad B_{kj}=\left(\frac{\partial A_{k}}{\partial x_{j}}- \frac{\partial A_{j}}{\partial x_{k}}\right) \quad k,j=1,\ldots,d.
\end{equation}
In geometric terms, it is given by the 2-form $dA$ as 
\begin{equation}\notag
dA = \sum_{k,j = 1}^{d} B_{kj}\, dx^{k}\wedge dx^{j}.
\end{equation}
In dimension $d=3$, $B$ is uniquely determined by the vector field $curl \, A$ via the vector product
\begin{equation}\notag
B v = curl \, A \times v, \quad \forall v \in \R^{3}.
\end{equation}
We also define the trapping component of $B$ as
\begin{equation}\label{tangen0}
B_{\tau}(x) = \frac{x}{|x|}B(x), \quad \quad \quad \quad (B_{\tau})_{j} = \sum_{k=1}^{d}\frac{x_{k}}{|x|} B_{kj}
\end{equation}
and we say that $B$ is non-trapping if $B_{\tau}=0$. Observe that in dimension $d=3$ it coincides with 
$$
B_{\tau}(x) := \frac{x}{|x|} \times curl \, A(x).
$$
Hence, $B_{\tau}(x)$ is the projection of $B$ on the tangential space in $x$ to the sphere of radius $|x|$, for $d=3$. Observe also that $B_{\tau} \cdot x = 0$ for any $d\geq 2$, therefore $B_{\tau}$ is a tangential vector field in any dimension and we call it the tangential component of the magnetic field $B$. 

In the sequel, we deal with potentials which vanish at infinity and are possibly singular at the origin. More precisely, we decompose the electric potential as
$$
V= V_{1} + V_{2},
$$
where $V_{1}$ is a long range potential and $V_{2}$ is a short range one which is possibly singular. Regarding to the magnetic part, some analogous conditions will be required for the magnetic field $B$, the quantity which is physically measurable. However, in order to ensure the self-adjointness of $L$ we need to require some local integrability condition on the magnetic potential $A$. From now on, we always assume that
\begin{equation}\label{assumptionself}
A_{j}\in L^{2}_{loc}(\Rd), \quad \quad V_{1}, V_{2} \in L^{1}_{loc}(\Rd), \quad \quad\int (V_{1}+V_{2})|u|^{2} \leq \nu\int |\nabla u|^{2}, \quad 0<\nu<1 .
\end{equation}
Thus it may be concluded (see \cite{Zu}, chapter 1 for more details) that $L$ is self-adjoint on $L^{2}(\Rd)$ with form domain
\begin{equation}\notag
D(L) = \{f\in L^{2}(\Rd) : \int |\D f|^{2} - \int (V_{1} + V_{2})|f|^{2} <\infty \}.
\end{equation} 
Note that by (\ref{assumptionself}) then $D(L)$ is equivalent to the Hilbert space 
$$
H_{A}^{1}(\Rd) = \{f\in L^{2}(\Rd) : \int |\D f|^{2} < \infty\}.
$$
As a consequence, since the spectrum of a self-adjoint operator is real, we obtain the existence of solution of the equation
\begin{equation}\label{resepsilon}
Lu_{\varepsilon} + (\lambda \pm i\varepsilon)u_{\varepsilon} =f
\end{equation}
in $\Rd$ for any $f\in L^{2}(\Rd)$ and $u_{\varepsilon}$ belonging to $H_{A}^{1}(\Rd)$. See \cite{IK}, \cite{LS}, \cite{AHS} or \cite{CFKS} for more details in the investigation of the essential self-adjointness of the electromagnetic Schr\"odinger operator $L$. 

Under suitable assumptions on the potentials, our goal is to prove that there exists a unique solution of the resolvent equation 
\begin{equation}\label{res}
(\nabla + iA)^{2}u +V_{1}u + V_{2}u + \lambda u =f, \quad \quad \lambda >0
\end{equation}
satisfying a specific Sommerfeld radiation condition together with some a priori estimates of Agmon-H\"ormander type. We will construct this solution $u$ from the solution of the equation (\ref{resepsilon}). In fact, $u$ will be the limit of $u_{\varepsilon}$ in a suitable space, that we will denote by
\begin{equation}\notag
u= R(\lambda + i0)f = \lim_{\varepsilon \to 0^{+}} u_{\varepsilon}.
\end{equation}
We point out that we need two main ingredients for this purpose. On the one hand, the a priori estimates and Sommerfeld radiation condition for any solution $u_{\varepsilon}\in H^{1}_{A}(\Rd)$ of (\ref{resepsilon}) will be needed. On the other hand, we shall assert uniqueness of solution of the equation (\ref{res}) if such a radiation condition is satisfied.

It is a simple matter to show the uniqueness result for (\ref{resepsilon}). Letting $f=0$, we only need to multiply the corresponding equation by $u_{\varepsilon}$ in the $L^{2}$-sense and take the imaginary part. Thus we get $\varepsilon\Vert u_{\varepsilon}\Vert^{2} =0$ and so $u_{\varepsilon}=0$. Uniqueness criterion for the equation (\ref{res}) presents a more delicate problem. In this case, we shall study the homogeneous electromagnetic Helmholtz equation
\begin{equation}\label{homogeneo}
(\nabla + iA)^{2}u + \left(V_{1} + V_{2}\right)u + \lambda u = 0
\end{equation}
and show that if $u\in (H^{1}_{A})_{loc}(\Rd)$ is a solution of $(\ref{homogeneo})$, then $u$ is identically zero. The proof of this result is adapted from \cite{M1} or \cite{Z}. Nevertheless, as far as we know, it does not seem to appear in the literature for potentials as the one we can treat. Using the multiplier method we prove that $u=0$ in $\Omega = \{x\in\Rd : |x|\geq R\}$ for $R>0$ large enough. Then we apply the unique continuation property to deduce that $u$ vanishes in $\Rd$. Hence, in order to accomplish this task, we need that the unique continuation property holds for $L$. 

Regbaoui \cite{R} proves that if $u\in H^{1}_{loc}(\Rd)$ satisfies
\begin{equation}\label{P(x,D)}
|P(x,D)u| \leq C_{1}|x|^{-2}|u| + C_{2}|x|^{-1}|\nabla u|,
\end{equation}
with $C_{2}>0$ small enough and $P(x,D)= \sum_{j,k=1}^{d}a_{jk}D_{j}D_{k}$ is an elliptic operator with Lipschitz coefficients such that $a_{jk}(0)$ is real in a connected open subset $\Omega$ of $\Rd$ containing $0$, then $u\equiv0$ in $\Omega$. Thus for using this result, we will write the magnetic Schr\"odinger operator $L$ as a first order perturbation of the Laplacian,
\begin{equation}\notag
L = \Delta  + 2i A\cdot\nabla  + i\nabla\cdot A  - A\cdot A + V_{1} + V_{2}
\end{equation}
and note that $u$ satisfies
\begin{equation}\label{laplreg}
|\Delta u + \lambda u| \leq 2|A||\nabla u| + \left(|\nabla \cdot A| + |V_{1}| + |V_{2}| + |A|^{2}\right)|u|.
\end{equation}

The crux of the limiting absorption principle are certain $L^{2}$-weighted a priori estimates for the operator $(L+ z)^{-1}$, $z=\lambda + i\varepsilon$, such that are preserved after the limiting procedure. The classical result on the free resolvent case, which is usually denoted by
\begin{equation}\notag
R_{0}(z) = (\Delta +z)^{-1},
\end{equation}
is due to Agmon \cite{A} and states that the limits
\begin{equation}\notag
R_{0}(\lambda \pm i0) = \lim_{\varepsilon \to 0} R_{0}(\lambda \pm i\varepsilon)
\end{equation}
exist in the norm of bounded operators from $L^{2}_{s}(\Rd)$ to $L^{2}_{-s}(\Rd)$ for any $s>1/2$, where
\begin{equation}\notag
\Vert u \Vert_{L^{2}_{s}} = \Vert (1+|x|)^{s}u\Vert_{L^{2}}.
\end{equation}
The convergence is uniform for $\lambda$ belonging to any compact subset of $]0+\infty[$, and the following estimate holds
\begin{equation}\label{Agmon}
\Vert R_{0}(\lambda \pm i0)f\Vert_{L^{2}_{-s}} \leq \frac{C(s)}{\sqrt{\lambda}} \Vert f \Vert_{L^{2}_{s}}, \quad \lambda>0, s>1/2.
\end{equation}
From this, it may be concluded that $u_{\pm}=R_{0}(\lambda \pm i0)f$ is the unique solution of the equation
\begin{equation}\notag
\Delta u_{\pm} + \lambda u_{\pm} = f
\end{equation}
satisfying
\begin{equation}\notag
\lim_{|x|\to +\infty} |x|^{\frac{d-1}{2}}\left(\frac{\partial u_{\pm}}{\partial |x|} \mp iku_{\pm} \right) = 0.
\end{equation}
$u_{+}= R_{0}(\lambda + i0)f$ is called the outgoing solution, while $u_{-}=R_{0}(\lambda - i0)f$ denotes the incoming one.

Later on, Agmon and H\"ormander \cite{AH} showed that estimate (\ref{Agmon}) held with the $L^{2}_{\delta}$ norms replaced by the norms
\begin{equation}\notag
|||u|||_{R_{0}} := \sup_{R>R_{0}} \left(\frac{1}{R}\int_{|x|\leq R} |u(x)|^{2}dx\right)^{1/2}
\end{equation}
and
\begin{equation}\notag
N_{R_{0}}(f) : = \sum_{j> J}\left(2^{j+1}\int_{C(j)} |f(x)|^{2}dx \right)^{1/2} + \left( R_{0}\int_{|x|\leq R_{0}} |f(x)|^{2}dx\right)^{1/2} 
\end{equation}
with $R_{0}=1$, where $C(j)=\{ x\in\Rd : 2^{j-1}\leq |x|\leq 2^{j}\}$ and $J$ is defined by $2^{J-1}\leq R_{0}\leq2^{J}$. 

The norms $|||u|||_{1}$ and $N_{1}(f)$ are known as Agmon-H\"ormander norms. We drop the index $R_{0}$ if $R_{0}=0$, getting then the Morrey-Campanato norm and its dual,
\begin{align}
& |||u|||:= \sup_{R>0} \left(\frac{1}{R}\int_{|x|\leq R} |u(x)|^{2}dx \right)^{1/2}\notag\\
& N(f):= \sum_{j\in \mathbb{Z}}\left( 2^{j+1}\int_{C(j)} |f(x)|^{2}dx\right)^{1/2}.\notag
\end{align}
Note that for all $R_{0} \geq 0$, it is satisfied
\begin{align}
\int fg &\leq \int_{|x|\leq R_{0}} |f||g| +  \sum_{j> \log_{2} R_{0}}\left[2^{j}\int_{C(j)} |f|^{2}\frac{1}{2^{j}}\int_{C(j)}|g|^{2} \right]^{1/2}\notag\\
&\leq |||g|||_{R_{0}}N_{R_{0}}(f).\notag
\end{align}

The Agmon-H\"ormander estimate was improved by Kenig, Ponce and Vega \cite{KPV} to the Morrey-Campanato norm in their study of the nonlinear Schr\"odinger equation. In fact, they proved
\begin{equation}\label{kpv}
\lambda^{1/2}|||u|||  \leq CN(f).
\end{equation}
This estimate plays a fundamental role in solving Schr\"odinger evolution equations with nonlinear first order terms.

The seminal papers by Agmon and H\"ormander \cite{A}, \cite{AH}, inspired a huge literature (see for example \cite{A}, \cite{AH}, \cite{Be}, \cite{Ik}, \cite{Is}, \cite{Mou} \cite{MU}, \ldots ) which has been produced in order to obtain weighted $L^{2}$-estimates for solutions of Helmholtz equations. Moreover, the classical work of Agmon \cite{A} shows the limiting absorption principle for short range perturbations of $\Delta$. Fourier analysis is involved as a crucial tool in the proofs strategy; however, Fourier transform does not permit in general to treat neither rough potentials nor the case in which the same problems are settled in domains that are different from the whole space. For this reason, a great effort has been spent in order to develop multiplier methods which work directly on the equation, inspired by the techniques introduced by Morawetz \cite{MO} for the Klein-Gordon equation.

Resolvent estimates for $\Delta + V$ with coefficients with very low regularity and such that $V$ does not vanish at infinity have been proved by Perthame and Vega \cite{PV1}, \cite{PV2}. The authors study the Helmholtz equation in an inhomogeneous medium of refraction index $n(x) = \lambda + V(x)$, generalizing the estimate (\ref{kpv}) to a variable case by using a multiplier method with appropriate weights as those used for the wave, Schr\"odinger or kinetic equations by Morawetz \cite{MO}, Lin and Strauss \cite{LS1} or Lions and Perthame \cite{LP}, respectively. We point out that the estimates are uniform for any $\lambda \geq 0$ and have the right scaling. Similar results but not scaling invariant were obtained in \cite{JP} and \cite{Zh1}. The scaling plays a fundamental role in the applications to nonlinear Schr\"odinger equation \cite{KPV} and in the high frequency limit for Helmholtz equations \cite{BCKP}, \cite{CPR}.

For the electromagnetic case, several papers are devoted to the study of the existence of a unique solution of the electromagnetic Helmholtz equation
\begin{equation}\label{electro}
(\nabla + iA(x))^{2} u + V(x) u + \lambda u = f(x), \quad x\in \Rd.
\end{equation}
The first result goes back to the work of Eidus \cite{E} in 1962, where it is showed that there exists a unique solution $u(\lambda, f)$ of the equation (\ref{electro}) in $\R^{3}$ with the radiation condition
\begin{equation}\notag
\lim_{r\to \infty} \int_{|x|=r} \left|\frac{\partial u}{\partial|x|}-i\lambda^{1/2}u \right|^{2} d\sigma(r)=0.
\end{equation}
Here $A_{j}(x)$ is assumed to vanish close to infinity and the electric potential satisfies $V(x)=O(|x|^{-2-\alpha})$ with $\alpha > \frac{1}{6}$ at infinity. 

In 1972, Ikebe and Saito \cite{IS} extend the above result to any $d\geq 3$ for potentials $V$ that are the sum of a long range potential $V_{1}$, being $V_{1}(x) = O(|x|^{-\mu})$, $\frac{\partial V_{1}}{\partial |x|} = O(|x|^{-1-\mu})$ at infinity and a short-range potential $V_{2}$ such that $V_{2}(x) = O(|x|^{-1-\mu})$, for $\mu >0$ when $|x|\to \infty$. Concerning the magnetic part, they require that $A_{j}\in C^{1}(\Rd)$ such that each component of the magnetic field holds $|B_{kj}|\leq C(1+|x|)^{-1-\mu}$ for some $C>0$, $\mu > 0$. By integration by parts they solve the electromagnetic Helmholtz equation (\ref{introISequ}) in a $L^{2}$-weighted space with the spherical radiation condition
\begin{equation}\label{IkebeSaito}
\int_{|x|\geq 1} \left|\nabla_{A} u - i\lambda^{1/2}\frac{x}{|x|}u\right|^{2} \frac{1}{(1+|x|)^{1-\delta}} < +\infty
\end{equation}
and a weighted $L^{2}$ a priori estimate
\begin{equation}\label{u00}
\int \frac{|u|^{2}}{(1+|x|)^{1+\delta}}   < +\infty,
\end{equation}
where $0< \delta < 1$ is a fixed constant. They require that the frequency $\lambda$ vary in a compact set $(\lambda_{0}, \lambda_{1})$ with $0<\lambda_{0} < \lambda_{1} < \infty$. This condition is essential to the justification of the compactness argument that they use in order to get (\ref{u00}). In \cite{IS} it is crucial to be far away from the zero frequency and the bounds they obtain are not uniform with respect to $\lambda \in[\lambda_{0,}\infty)$.

The literature about resolvent estimates related to the magnetic Schr\"odinger operator is more extensive. We are mainly interested in giving a priori estimates for solutions $u$ of the equation (\ref{introISequ}) imposing conditions on the trapping component of the magnetic field $B$, instead of on the magnetic potential $A$. The quantity $B_{\tau}$ was introduced by Fanelli and Vega \cite{FV} in which it is proved that weak dispersion for the magnetic Schr\"odinger and wave equation holds, for example, for non-trapping potentials, i.e., $B_{\tau}=0$. This is also what happens in the stationary case, as it is shown in \cite{F}. Following \cite{PV1}, Fanelli generalizes the uniform a priori estimate (\ref{kpv}) to the magnetic case. This estimate has several consequences about the so called Kato smoothing effects for solutions of the evolutions problems which in general do not hold for long range potentials (see among others \cite{BRV}, \cite{DF}, \cite{Ka4}, \cite{KaYa}, \cite{LP}, \cite{LS1}). The uniform resolvent estimate
\begin{equation}\notag
\int \frac{|u|^{2}}{|x|^{2}} \leq C\int |x|^{2}|f|^{2}
\end{equation}
also plays a fundamental role for dispersive estimates on the time dependent Schr\"odinger operator, as for the study of the Strichartz estimates for the Schr\"odinger equation with electromagnetic potential, see for example \cite{DFVV}, \cite{FG}, \cite{M2}, \cite{M3}.

In this paper we are able to strongly improve the result by Ikebe and Saito \cite{IS}, inspired by the multiplier technique introduced in \cite{PV1}. Let us consider the inhomogeneous Helmholtz equation
\begin{equation}
(\nabla +iA)^{2}u + V_{1}u + V_{2}u + \lambda u +i\varepsilon u = f,  \label{2.1}
\end{equation}
where $\lambda, \varepsilon > 0$ and $f$ is a suitable function on $\Rd$. We work with potentials that decay at infinity and can have singularities at a point that we will take to be at the origin. We will use a multiplier method based on radial multipliers. Thus just information for the tangential component of the magnetic field $B$ (see Remark \ref{Btauremark}) will be needed. Nevertheless, in order to assert the unique continuation property, it is necessary to put some restrictions on the whole $B$ when we are close to the origin. 

One question still unanswered is whether the unique continuation property is satisfied assuming only the decay on the tangential part of $B$. We will not develop this point here, but we propose to study it in the future. We should mention here that a partial result for $B_{\tau}=0$ has been obtained. We refer the reader to \cite{Zu}.

Before stating our main result, we need some preliminaries. From now on, we denote the radial derivative and the tangential component of the magnetic gradient $\D$ defined in (\ref{1.4}) as 
\begin{equation}\notag
\nabla_{A}^{r}u = \frac{x}{|x|}\cdot \nabla_{A}u, \quad \quad |\nabla_{A}^{\bot}u|^{2}=|\nabla_{A}u|^{2}-|\nabla_{A}^{r}u|^{2},
\end{equation}
respectively. Moreover, we recall (see \cite{LL}) the diamagnetic inequality 
\begin{equation}\label{diamagnetic}
|\nabla |f|(x)| \leq |\D f(x)|,
\end{equation}
which holds pointwise for almost every $x\in \Rd$ and for any $f\in H^{1}_{A}(\Rd)$, if $A\in L^{2}_{loc}(\Rd)$.

We assume that the magnetic potential $A$ satisfy 
\begin{equation}\label{gradienteb}
|\nabla \cdot A| \leq \frac{C}{|x|^{2}},
\end{equation}
for some $C>0$. We point out that this condition is only needed for the unique continuation property. In addition, we will require that 
\begin{equation}\label{extracondition}
\int_{|x|\leq R} |Au|^{2} \leq C_{R}\int |\nabla u|^{2}
\end{equation}
for any $R>0$ and some $C_{R}>0$. Combining this condition with the diamagnetic inequality (\ref{diamagnetic}) and Cauchy-Schwarz inequality, since
\begin{align}\notag
|\nabla u|^{2} &= |\D u|^{2} - |Au|^{2} + 2\Im A\bar{u}\cdot \nabla u\notag\\
& = |\D u|^{2} -3|Au|^{2} + 2\Im A\cdot \D u\bar{u},\notag 
\end{align}
we conclude that
\begin{align}
\int_{|x|\leq R} |\nabla u|^{2} 
%&\leq \int_{|x|\leq R} |\D u|^{2} +3\int_{|x|\leq R} |Au|^{2} +2\left(\int_{|x|\leq R} |Au|^{2} \right)^{\frac{1}{2}}\left(\int_{|x|\leq R} |\D u|^{2} \right)^{\frac{1}{2}}\notag\\
& \leq C \int |\D u|^{2}.\notag
\end{align}
As a consequence, if $\D u \in L^{2}(\Rd)$ then $\nabla u \in L^{2}_{loc}(\Rd)$. Condition (\ref{extracondition}) will be used for the compactness argument.

We may now state our main assumptions on the potentials.

%\textbf{Assumptions.}
\begin{ass}\label{ass1}
\emph{Let $V_{1}(x)$, $A_{j}(x)$, $j=1,\ldots,d$, $V_{2}(x)$ be real-valued functions, $r_{0} \geq 1$ and $\mu>0$. For $d\geq 3$, if $|x|\geq r_{0}$ we assume
\begin{equation}\label{condicionf0}
\frac{|V_{1}(x)|}{|x|} + (\partial_{r}V_{1}(x))_{-} + |B_{\tau}(x)| + |V_{2}(x)| \leq \frac{c}{|x|^{1+\mu}},
\end{equation}
for some $c >0$, where $\partial_{r}V_{1}= \frac{x}{|x|}\cdot \nabla V_{1}$ is considered in the distributional sense and $(\partial_{r}V_{1})_{-}$ denotes the negative part of $\partial_{r}(V_{1})$. On the other hand, we require
\begin{equation}\label{(p0)}
V_{1}(x)=(\partial_{r}V_{1}(x))_{-}=0 \quad \textrm{if} \quad |x|\leq r_{0},
\end{equation}
and
\begin{equation}\label{(Q0)}
|V_{2}(x)|\leq \frac{c}{|x|^{2-\alpha}} \quad \textrm{if} \quad |x|\leq r_{0}, \quad \alpha >0,
\end{equation}
for some $c>0$. }
\\
\emph{If $d>3$, we consider
\begin{equation}\label{(b0)}
|B| \leq \frac{\mathcal{C^{*}}}{|x|^{2}} \quad  \quad |x|\leq r_{0},
\end{equation}
for some $\C>0$ small enough. Finally, in dimension $d=3$ we assume
\begin{equation}\label{(b10)}
|B| \leq \frac{c}{|x|^{2-\alpha}} \quad  \quad |x|\leq r_{0}, \quad \alpha > 0,
\end{equation}
for some $c>0$.}
\end{ass}

\begin{remark}\label{r01}
Without loss of generality and for simplicity, throughout the paper we take $r_{0}=1$.
\end{remark}
 
\begin{remark}
Note that the requirements on the magnetic field $B$ at the origin differ depending on the dimension. This is due to the fact that we give an extra a priori estimate for the solution $u$ of the equation (\ref{res}) when $d>3$, see (\ref{i0}) below.
\end{remark}

\begin{remark}\label{apendice}
For $d>3$ we may allow the potential $V_{2}(x)$ to be more singular. Moreover, we can also permit some singularity on the potential $V_{1}(x)$ and its repulsive part $(\partial_{r}V_{1}(x))_{-}$. When $|x|\leq r_{0}$, one can actually require
\begin{equation}\notag
|V_{2}(x)| \leq \frac{\mathcal{C}^{**}}{|x|^{2}}
\end{equation}
and
\begin{equation}\notag
(\partial_{r}V_{1}(x))_{-} \leq \frac{\mathcal{C}^{**}}{|x|^{3}}, \quad \frac{|V_{1}(x)|}{|x|} \leq \frac{\mathcal{C}^{***}}{|x|^{3}}
\end{equation}
for sufficiently small $\mathcal{C}^{**}$$ > 0$ and for some $\mathcal{C}^{***}>0$. See \cite{Zu}, chapter 2 for more details.
\end{remark}

\begin{remark}
Observe that in order to use the unique continuation result (\cite{R}), by (\ref{laplreg}) we need to verify that
\begin{equation}\label{divergencia}
|\nabla \cdot A| \leq C_{1}|x|^{-2}
\end{equation}
and
\begin{equation}\label{potencialb}
|A| \leq C_{2}|x|^{-1}
\end{equation}
provided that $C_{2}>0$ is small. On the one hand, note that condition (\ref{gradienteb}) gives (\ref{divergencia}). On the other hand, from (\ref{(b0)}) when $d>3$ with $\mathcal{C}^{*}$ small enough and (\ref{(b10)}) when $d=3$, by the Biot-Savart law it may be concluded that (\ref{potencialb}) holds. It is worth pointing out that condition (\ref{gradienteb}) is only required in order to assure that this result is applicable. 

\end{remark}

Our first theorem is the uniqueness result.

\begin{thm}\label{unicidad1}
Let $d \geq 3$, $\lambda \geq \lambda_{0} >0$ and assume (\ref{gradienteb}), (\ref{condicionf0})-(\ref{(Q0)}) and (\ref{(b0)}) or (\ref{(b10)}). Let $u\in (H^{1}_{A})_{loc}(\Rd)$ be a solution of (\ref{homogeneo}) such that 
\begin{equation}
\liminf \int_{|x|= r} (|\D u|^{2} + \lambda|u|^{2}) d\sigma(x) \to 0, \quad \textrm{as} \quad r \to \infty. \label{1.130}
\end{equation}
Then $u \equiv 0$. Moreover, if for some $\delta>0$
\begin{equation}
\int_{|x|\geq 1} \left|\D u - i\lambda^{1/2}\frac{x}{|x|}u\right|^{2}\frac{1}{(1+|x|)^{1-\delta}} < \infty
\end{equation}
is satisfied, then (\ref{1.130}) holds.
\end{thm}

The uniqueness result allows us to state the main result of this paper.

\begin{thm}\label{d>30}
Let $\mathcal{C}^{*}$ small enough, $\lambda_{0} >0$, $f\in L^{2}_{\frac{1+\delta}{2}}(\Rd)\cap L^{2}_{\delta}(\Rd)$ and assume that one of the following two conditions is satisfied:
\begin{itemize}
\item[(i)] $d >3$, with (\ref{gradienteb})-(\ref{(Q0)}) and (\ref{(b0)})
\item[(ii)] $d=3$, with (\ref{gradienteb})-(\ref{(Q0)}) and  (\ref{(b10)}).
\end{itemize}
Then, for all $\lambda\geq \lambda_{0}$ there exists a unique solution $u\in (H^{1}_{A})_{loc}(\Rd)$ of the Helmholtz equation (\ref{res}) satisfying 
\begin{align}\label{i0}
&\lambda|||u|||_{1}^{2} + |||\D u|||_{1}^{2} + \int \frac{|\D^{\bot}u|^{2}}{|x|} + \sup_{R>0}\frac{1}{R^{2}}\int_{|x|=R} |u|^{2}d\sigma_{R}\\
& + (d-3)\int \frac{|u|^{2}}{|x|^{3}} \leq C (N_{1}(f))^{2}\notag
\end{align}
and the radiation condition
\begin{equation}\label{radiacion0}
\int_{|x|\geq 1}  \left|\D u- i\lambda^{1/2}\frac{x}{|x|}u\right|^{2} \frac{1}{(1+|x|)^{1-\delta}} \leq C \int (1+|x|)^{1+\delta}|f|^{2},
\end{equation}
for all $0<\delta< 2$ such that $\delta<\mu$, where $C=C(\lambda_{0})>0$.
\end{thm}

\begin{remark}
The smallness of the constant $\mathcal{C}^{*}$ is required for the unique continuation property proved by Regbaoui \cite{R}. This constant is not explicit.
\end{remark}

%\begin{remark}
%The case $B_{\tau}=0$ is particularly interesting as we saw in section \ref{exapotentials}. In the Appendix B, Theorem \ref{btauzero}, we give a unique continuation result for this %kind of potentials. As a consequence, if $V_{1} = V_{2} = 0$ Theorems \ref{unicidad1} and \ref{d>30} will hold in this case.
%\end{remark}

\begin{remark}
In order to prove the a priori estimate (\ref{i0}), condition (\ref{(b0)}) can be replaced by
\begin{equation}\notag
|B_{\tau}| \leq \frac{(d-1)(d-3)}{|x|^{2}}, \quad \quad |x|\leq r_{0}.
\end{equation}
\end{remark}

%\begin{remark}
%The smallness of the constant $\mathcal{C}^{*}$ is required not only for the unique continuation property, but also for proving the Morrey-Campanato type estimates (\ref{i0}), in %which we are able to give the explicit constant (see Theorem \ref{landagrande0} and condition (\ref{C^{*}}) below) only assuming the decay of the tangential part of $B$. One %may ask whether this constant is sufficient in order to ensure the unique continuation property for $L$. 
%\end{remark}

Theorem \ref{d>30} extends the result proved by Ikebe and Saito in the 70's. Firstly, our estimates are not only true for $\lambda\in (\lambda_{0}, \lambda_{1})$ with $0<\lambda_{0} < \lambda_{1} <\infty$ as in \cite{IS}, but also for all $\lambda \geq \lambda_{0} >0$. We also extend the Sommerfeld radiation condition (\ref{IkebeSaito}) from $\delta \in (0, 1)$ to the range $\delta\in (0,2)$. Concerning the a priori estimates, note that (\ref{i0}) is stronger than (\ref{u00}) in the sense that it gives more information about the solution and improves the $L^{2}$-weighted estimate from the $L^{2}_{-\frac{(1+\delta)}{2}}$ norm to the Agmon-H\"ormander norm. More importantly, we are able to consider singular potentials and the estimate (\ref{i0}) is uniform on $\lambda$ for $\lambda\geq \lambda_{0} >0$. This permits to prove the $L^{p}$-$L^{q}$ estimates for the electromagnetic Helmholtz equation with singular potentials. (See \cite{G1}, chapter 2 and \cite{G}). 

In order to recover the a priori estimates in the full frequency range $\lambda\geq 0$, a stronger decay on the potentials is needed. In 2009, Fanelli \cite{F} proved (\ref{i0}) with the Agmon-H\"ormander norm replaced by the Morrey-Camapanato one for any $\lambda \geq 0$ in $\Rd$. Very recently, Barcel\'o, Fanelli, Ruiz and Vilela \cite{BLRV} also get the analogous estimates for the Helmholtz equation with electromagnetic-type perturbations in the exterior of a domain. In fact, if such an estimate holds for $\lambda \geq 0$, it would imply as a by product the absence of zero-resonances (in a suitable sense) for the operator $L$. This is in general false with our type of potentials. For example, if we reduce to the case $\Delta u + Vu = 0$, the natural decay at infinity for the non-existence of zero-resonances is $|x|^{-(2+\delta)}$, $\delta >0$. See for example \cite{BRV1}, Section 3 and \cite{F} Remark 1.3. 

The general outline for proving the main result consists of the following steps:

\begin{itemize}
\item[1.] We take a sufficiently large $\lambda_{1} (> \lambda_{0})$ and we derive the Agmon-H\"ormander type estimates for any $\lambda\geq\lambda_{1}$ proceeding as in \cite{PV1}.
\item[2.] We prove that for any $\lambda\geq\lambda_{0}$ the Sommerfeld radiation condition is true if the Agmon-H\"ormander type estimates hold.
\item[3.] We use a compactness argument (in the spirit of \cite{IS}) to deduce the result for all $\lambda\geq\lambda_{0}$.
\item[4.] From the estimates proved in the previous steps and by the uniqueness theorem, we prove the limiting absorption principle for the Schr\"odinger operator $L$ satisfying (\ref{assumptionself}), (\ref{gradienteb})-(\ref{(b10)}).
\end{itemize}

\vspace{0.1cm}
{\bf Notation.} Throughout the paper, $C$ denotes an arbitrary positive constant and $\kappa$ stands for a small positive constant. In most of the cases, $\kappa$ will come from the inequality $ab \leq \kappa a^{2}+ \frac{1}{4\kappa}b^{2}$, which is true for arbitrary $\kappa >0$. In the integrals where we do not specify the integration space we mean that we are integrating in the whole $\Rd$ with respect to the Lebesgue measure $dx$, i.e. $\int = \int_{\Rd} dx$.

\section{Proof of Theorem \ref{d>30}}\label{sectionmain}

According to the steps given above, the proof will be divided into four parts. 

\subsection{A priori estimates for $\lambda$ large enough ($\lambda \geq \lambda_{1}$)}

We begin by proving the Agmon-H\"ormander type estimates for solutions of the equation (\ref{2.1}) for $\lambda$ large enough. Since our assumptions on the magnetic field differ depending on the dimension, we first give a detailed proof of the result for $d>3$. Then the three dimensional case follows by the same method.

\begin{thm}\label{landagrande0}
For dimension $d>3$, let $\varepsilon >0$, $f$ such that $N_{1}(f)<\infty$. Let $\C<\sqrt{(d-1)(d-3)}$. Assume that (\ref{condicionf0})-(\ref{(Q0)}) and 
\begin{equation}\label{(b0'')}
|B_{\tau}|\leq \frac{\mathcal{C}^{*}}{|x|^{2}} \quad \quad  \textrm{if} \quad \quad |x|\leq 1
\end{equation}
hold. Then there exists $\lambda_{1}>0$ such that for any $\lambda \geq \lambda_{1}$ the solution $u\in H^{1}_{A}(\Rd)$ of the Helmholtz equation (\ref{2.1}) satisfies
\begin{align}\label{landabig0}
\lambda|||u|||_{1}^{2} + |||\D u|||_{1}^{2} + \int &\frac{|\D^{\bot}u|^{2}}{|x|} + \sup_{R > 0}\frac{1}{R^{2}}\int_{|x|=R} |u|^{2}d\sigma_{R} + \int \frac{|u|^{2}}{|x|^{3}}\\
& \leq C(1+\varepsilon)(N_{1}(f))^{2},\notag
\end{align}
where $C=C(\lambda_{1})>0$ is independent of $\varepsilon$.
\end{thm}

\begin{proof}
The proof is based on the identities which are established in Appendix. Let $\varphi, \psi$ be real-valued radial functions. Adding up (\ref{(4.11)}) and (\ref{(4.3)}) we have
\begin{align}\label{batuketa}
&\int \D u\cdot D^{2}\psi\cdot  \overline{\D u}-\int \varphi |\D u|^{2}  - \Re \int \left(\nabla\varphi -\frac{\nabla(\Delta\psi)}{2} \right)\cdot \D u\bar{u}\\ 
&+\int\lambda|u|^{2}+ \int \varphi V_{1} |u|^{2}+ \frac{1}{2}\int \psi' \partial_{r} V_{1} |u|^{2}- \Im  \int \psi' B_{\tau} \cdot \D u \bar{u}\notag\\
& + \int \left( \varphi  -\frac{ \Delta\psi}{2}\right) V_{2}|u|^{2} - \Re\int V_{2}\nabla\psi \cdot \D u \bar{u}\notag\\
& = \varepsilon \Im\int \nabla\psi\cdot \D u\bar{u}+\Re\int\left( \varphi-\frac{ \Delta\psi}{2}\right) f\bar{u}-\Re \int f\nabla\psi \cdot \overline{\D u}. \notag
\end{align}

Let us define for $R>0$ the function $\psi(x)= \int_{0}^{|x|} \phi(s)ds$, where
\begin{equation}\notag
\phi(r)  = \left\{ \begin{array}{ll}
\frac{r}{2R} + M & \textrm{if $0 < r \leq R$},\vspace{0.1cm} \\
M+\frac{1}{2} & \textrm{if $r\geq R$},
\end{array} \right.
\end{equation}
so that $\nabla \psi(x) = \frac{x}{|x|}\phi(|x|)$ for arbitrary $M>0$,
\begin{equation}\notag
\varphi(x) = \left\{ \begin{array}{ll}
\frac{1}{4R} & \textrm{if $|x| \leq R$},\\
0 & \textrm{if $|x| \geq R$},
\end{array} \right.
\end{equation}
and we put these multipliers into (\ref{batuketa}). 

First, note that since $N_{1}(f)<\infty$ then $f\in L^{2}(\Rd)$. Thus it is guaranteed the existence of solution of (\ref{2.1}) in $H_{A}^{1}(\Rd)$. As a consequence, the terms on the right-hand side of (\ref{batuketa}) are finite. It is easy to check that 
\begin{align}
&\left| \varepsilon \Im\int \nabla\psi\cdot \D u\bar{u}+\Re\int\left( \varphi-\frac{ \Delta\psi}{2}\right) f\bar{u}-\Re \int f\nabla\psi \cdot \overline{\D u}\right| \notag\\
& \leq C \left(\Vert f\Vert_{L^{2}}^{2} + \Vert \D u\Vert_{L^{2}}^{2} + \Vert u \Vert_{L^{2}}^{2} \right) < \infty.\notag
\end{align}

Let us show the positivity of the left-hand side of (\ref{batuketa}) with the above choice of the multipliers. Since
\begin{equation}\label{hessianoa}
\D u \cdot D^{2}\psi \cdot \overline{\D u} = \psi''|\D^{r}u|^{2} + \frac{\psi'}{r}|\D^{\bot}u|^{2},
\end{equation}
it follows easily that
\begin{align}\notag
\int \D u\cdot D^{2}\psi \cdot\overline{\D u} - \int \varphi|\D u|^{2} & > \frac{1}{4R}\int_{|x|\leq R} |\D u|^{2} + M\int\frac{|\D^{\bot}u|^{2}}{|x|},\notag
\end{align}
\begin{align}\notag
\int \varphi \lambda |u|^{2} = \frac{1}{4R}\int_{|x|\leq R} \lambda |u|^{2}.
\end{align}
In addition, since $\varphi$ and $\psi''$ are discontinuous in $\{|x|=R\}$, note that integrating by parts the term
\begin{equation}\label{phipsi}
-\Re\int \left(\nabla\varphi - \frac{\nabla(\Delta\psi)}{2} \right)\cdot \D u\bar{u}
\end{equation}
gives a surface integral. In fact, after substituting our test functions in (\ref{phipsi}), we get
\begin{align}
-\Re\int \left(\nabla\varphi- \frac{\nabla(\Delta\psi)}{2} \right)\cdot \D u\bar{u} & > \frac{M(d-1)(d-3)}{4}\int \frac{|u|^{2}}{|x|^{3}}\notag\\
& + \frac{(d-1)}{8R^{2}}\int_{|x|=R}|u|^{2}d\sigma_{R}.\notag
\end{align}

Let us analyze the terms containing the potentials. In what follows, $\sigma=\sigma(c,\mu,\alpha, M)$ denotes a positive constant where the parameters $c, \mu, \alpha$ have been introduced in Assumption \ref{ass1} and $M>0$ is related to the multipliers. For simplicity of notation, we use the same letter $\sigma$ for all constants related to the potentials. 

In order to estimate the term involving the magnetic field, observe that since $B_{\tau}$ is a tangential vector to the sphere, we have
\begin{equation}\label{magnetictangenbtau}
B_{\tau} \cdot \D u = B_{\tau} \cdot \D^{\bot} u.
\end{equation}
Hence, 
\begin{align}
\Im  \int \psi' B_{\tau} \cdot \D u \bar{u} &\leq (M+1/2) \int_{|x|\leq 1}  |B_{\tau}||\D^{\bot}u| |u|\notag\\
& +(M+1/2) \int_{|x|\geq 1} |B_{\tau}||\D^{\bot}u| |u|\notag\\
&\equiv B_{1} + B_{2}\notag,
\end{align}
where by (\ref{condicionf0}), (\ref{(b0'')}) and Cauchy-Schwarz inequality, yields
\begin{align}
B_{1} & \leq \C(M+1/2)\left( \int_{|x|\leq 1}\frac{|\D^{\bot}u|^{2}}{|x|} \right)^{1/2}\left( \int_{|x|\leq 1} \frac{|u|^{2}}{|x|^{3}} \right)^{1/2}\label{magnetico}
\end{align}
and
\begin{align}
B_{2} & \leq (M+1/2)\int_{|x|\geq 1}|x|^{1/2}|B_{\tau}|\frac{|\D^{\bot}u||u|}{|x|^{1/2}}\notag\\
& \leq \frac{M}{2}\int_{|x|\geq 1}\frac{|\D^{\bot}u|^{2}}{|x|} + \frac{c(M+1/2)^{2}}{2M}\sum_{j\geq 0}2^{-2\mu j}\int_{C(j)}\frac{|u|^{2}}{2^{j}}\notag\\
& \leq \frac{M}{2}\int_{|x|\geq 1}\frac{|\D^{\bot}u|^{2}}{|x|} + \sigma|||u|||_{1}^{2}.\notag
\end{align}
We next turn to estimate the $V_{1}$ terms. Similarly, by (\ref{condicionf0}) and (\ref{(p0)}), we get
\begin{align}
-\int \varphi V_{1}|u|^{2} & \leq \frac{1}{4}\int_{1\leq |x| \leq R} \frac{|V_{1}||u|^{2}}{|x|}\notag\\
& \leq \frac{1}{4}\sum_{j\geq 0}\int_{C(j)} \frac{|V_{1}||u|^{2}}{2^{j}} \leq \sigma|||u|||_{1}^{2}\notag
\end{align}
and
\begin{align}
-\frac{1}{2}\int \psi' \partial_{r}V_{1}|u|^{2} &\leq \frac{1}{2}\int \psi'(\partial_{r} V_{1})_{-}|u|^{2}\notag\\
& \leq \frac{(M+1/2)}{2}\sum_{j\geq 0}\int_{C(j)}(\partial_{r} V_{1})_{-}|u|^{2}\notag\\
& \leq \sigma|||u|||_{1}^{2}.\notag
\end{align}
As far as the potential $V_{2}$ is concerned, let us first take $j_{1}=j_{1}(\alpha) < 0$ such that
\begin{equation}\label{alpha}
c\sum_{j\leq j_{1}} 2^{\alpha j} < \eta, \quad \quad \quad \quad c\sum_{j\leq j_{1}} 2^{2\alpha j} < \eta
\end{equation}
where $\eta>0$ stands for a small constant, being $c$ and $\alpha$ as in (\ref{(Q0)}). To simplify notation, we continue to write $\eta$ for any small constant related to the potentials. We fix $r_{1}<1$ by $2^{j_{1}-1}\leq r_{1} \leq 2^{j_{1}}$. Then, by (\ref{condicionf0}), $(\ref{(Q0)})$ and Cauchy-Schwarz inequality, we have
\begin{align}
\Re\int V_{2}\nabla\psi\cdot\D u \bar{u} & \leq (M+1/2) \int_{|x|\leq 1} |V_{2}||\D u| |u|\notag\\
& +(M+1/2) \int_{|x|\geq 1} |V_{2}||\D u||u|\notag\\
& \equiv V_{21} + V_{22}.\notag
\end{align}
Let us make now the following observation. 
\begin{align}
\sum_{j\leq 0}\int_{C(j)}\frac{|u|^{2}}{2^{j(3-\gamma)}} & \leq \sum_{j\leq 0}\int_{2^{j-1}}^{2^{j}}\int_{|x|=r} \frac{|u|^{2}}{r^{2}} 2^{-j(1-\gamma)}\notag\\
& \leq \sup_{R\leq 1} \frac{1}{R^{2}}\int_{|x|=R} |u|^{2}\sum_{j\leq 0} \int_{2^{j-1}}^{2^{j}} 2^{j(\gamma-1)}\notag\\
& \leq \sup_{R >0} \frac{1}{R^{2}}\int_{|x|=R} |u|^{2}\sum_{j\leq 0} 2^{j\gamma}\notag
\end{align}
and $\sum_{j\leq 0} 2^{\gamma j} < \infty$ if $\gamma >0$. According to the above remark, using again the Cauchy-Schwarz inequality and the relation $ab \leq \frac{1}{16} a^{2} +4 b^{2}$, we have
\begin{align}
V_{21} &\leq c(M+1/2)\left[\sum_{j\leq j_{1}}\int_{C(j)}\frac{|\D u||u|}{2^{j(2-\alpha)}} + \sum_{j=j_{1}}^{0}2^{-j(1-\alpha)}\int_{C(j)}\frac{|\D u||u|}{2^{j}}\right] \notag\\
&  \leq \eta\left(\sup_{R\leq r_{1}}\frac{1}{R}\int_{|x|\leq R}|\D u|^{2} \right)^{\frac{1}{2}}\left(\sup_{R>0}\frac{1}{R^{2}}\int_{|x|=R}|u|^{2} \right)^{\frac{1}{2}}\notag\\
&+\frac{1}{16}\sup_{R>0} \frac{1}{R} \int_{r_{1} \leq |x|\leq R} |\D u|^{2} + \sigma|||u|||_{1}^{2}\notag
\end{align}
and
\begin{align}
V_{22} & \leq c(M+1/2)\sum_{j\geq 0}\left( \int_{C(j)} \frac{|\D u|^{2}}{2^{j}} \right)^{1/2}\left( \int_{C(j)}\frac{|u|^{2}}{2^{j(1+2\mu)}}\right)^{1/2}\notag\\
& \leq \frac{1}{16}\sup_{R\geq 1} \frac{1}{R} \int_{r_{1}\leq |x|\leq R} |\D u|^{2}  + \sigma|||u|||_{1}^{2}.\notag
\end{align}
Analysis similar to the above gives
\begin{align}
-\int \varphi V_{2}|u|^{2} & \leq \frac{c}{4}\int_{|x|\leq 1}\frac{|u|^{2}}{|x|^{3-\alpha}} + \frac{1}{4}\sum_{j\geq 0}\int_{C(j)}\frac{|V_{2}||u|^{2}}{|x|}\notag\\
& \leq \eta\sup_{R \leq r_{1}} \frac{1}{R^{2}}\int_{|x|=R}|u|^{2}d\sigma_{R} +\sigma|||u|||_{1}^{2}\notag
\end{align}
and
\begin{align}
\frac{1}{2}\int \Delta\psi V_{2}|u|^{2} & \leq \left(\frac{d}{4}+\frac{M(d-1)}{2}\right)\int_{\Rd}\frac{|V_{2}||u|^{2}}{|x|}\notag\\
& \leq \eta\left(\frac{d}{4}+\frac{M(d-1)}{2}\right)\sup_{R >0} \frac{1}{R^{2}}\int_{|x|=R}|u|^{2}d\sigma_{R}\notag\\
&+\sigma|||u|||_{1}^{2}. \notag
\end{align}

In order to simplify the reading, let us introduce
\begin{align*}
&a_{1}=\left(\sup_{R\leq r_{1}}\frac{1}{R}\int_{|x|\leq R}|\D u|^{2}\right)^{1/2} , \qquad a_{2}=\left( \int_{|x|\leq 1}\frac{|\D^{\bot}u|^{2}}{|x|} \right)^{1/2},\\
&a_{3}=\left( \int_{|x|\leq 1}\frac{|u|^{2}}{|x|^{3}}\right)^{1/2}, \qquad \quad a_{4}=\left(\sup_{R>0}\frac{1}{R^{2}}\int_{|x|=R}|u|^{2}d\sigma_{R}\right)^{1/2}.
\end{align*}
Therefore, it turns out that the potential terms on (\ref{batuketa}) are lower bounded by
\begin{align}
&-\frac{M}{2}\int_{|x|\geq 1} \frac{|\D^{\bot}u|^{2}}{|x|} - \mathcal{C}^{*}(M+1/2)a_{2}a_{3} -(M+1/2)\eta a_{1}a_{4}- \eta a_{4}^{2}\notag\\
& -\frac{1}{8}\sup_{R>0}\frac{1}{R}\int_{r_{1}\leq |x|\leq R} |\D u|^{2} - \sigma|||u|||_{1}^{2}\notag.
\end{align}

Our next step is to estimate the right-hand side of (\ref{batuketa}). Let us start by the $\varepsilon$ term. From the a priori estimate (\ref{b+++0}), by the assumptions (\ref{condicionf0})-(\ref{(Q0)}), by (\ref{alpha}) and the Hardy inequality (\ref{A.1}), it may be concluded that
\begin{equation}\label{B10}
\int |\D u|^{2} \leq \lambda\int |u|^{2} +\sigma\int |u|^{2} + \int_{\Rd}|f||u|.
\end{equation}
Recall that $\sigma$ denotes a positive constant related to the potentials. Hence combining (\ref{B10}) with (\ref{a+++0}), by Cauchy-Schwarz inequality and the fact that $\int |f||u| \leq N_{1}(f)|||u|||_{1}$, we obtain

\begin{align}\label{epsilonterm1}
\varepsilon \Im\int\nabla\psi\cdot\D u\bar{u} &\leq (M+1/2)\varepsilon\int |\D u||u|\\
& \leq (M+1/2)\varepsilon^{1/2}\left(\varepsilon\int |u|^{2}\right)^{1/2}\left(\int |\D u|^{2}\right)^{1/2}\notag\\
& \leq (M+1/2)\varepsilon^{1/2}\int |f||u|\notag\\
& +(M+1/2)\varepsilon^{\frac{1}{2}} \left(\int |f||u|\right)^{\frac{1}{2}}\left( (\lambda +\sigma)\int |u|^{2}\right)^{\frac{1}{2}}\notag\\
& \leq (M+1/2)(\varepsilon^{1/2} + (\lambda+\sigma)^{1/2})\int |f||u|\notag\\
& \leq \kappa(1+\lambda)|||u|||_{1}^{2} + C(1+\varepsilon)(N_{1}(f))^{2}.\notag
\end{align}
It remains to estimate the terms containing $f$ which can be handled in much the same way as the rest. In fact, it follows that
\begin{align*}
&\Re\int  f\left( \varphi - \frac{1}{2}\Delta\psi\right)\bar{u}  \leq \frac{M+1}{2}\left( \int_{|x|\leq 1} \frac{|f| |u|}{|x|} + \int_{|x|\geq 1} \frac{|f||u|}{|x|}\right)\\
&\leq \frac{M+1}{2}\left[\left(\int_{|x|\leq 1} |f|^{2} \right)^{\frac{1}{2}}\left(\int_{|x|\leq 1} \frac{|u|^{2}}{|x|^{3}}\right)^{\frac{1}{2}} +\sum_{j\geq 0} \left( 2^{j}\int_{C(j)} |f|^{2}\right)^{\frac{1}{2}}\left( \int_{C(j)}\frac{|u|^{2}}{2^{3j}}\right)^{\frac{1}{2}}\right]\\
& \leq \kappa \sup_{R>0}\frac{1}{R^{2}}\int_{|x|=R}|u|^{2}d\sigma_{R} + C(N_{1}(f))^{2}
\end{align*}
and
\begin{align*}
\Re\int f\nabla\psi\cdot\overline{\D u} &\leq (M+1/2)\int  |f||\D u| \notag\\
&\leq \kappa|||\D u|||_{1}^{2} + C(N_{1}(f))^{2}.\notag
\end{align*}

Finally, due to the freedom on the choice of $R$, let us take the supremum over $R>0$ on the both sides of the inequality. Thus from the above estimates, we obtain
\begin{align}
& \left( \frac{\lambda}{4}-\sigma\right) |||u|||^{2}  + \frac{M}{2}\int_{|x|\geq 1}\frac{|\D^{\bot}u|^{2}}{|x|}+  \left(Ê\frac{d-1}{8} - \eta\right) a_{4}^{2}+\frac{a_{1}^{2}}{4}  - \eta a_{1}a_{4}\notag\\
& +  \frac{1}{8}\sup_{R>0}\frac{1}{R}\int_{r_{1}\leq |x|\leq R} |\D u|^{2} + \frac{M(d-1)(d-3)}{4}\int_{|x|\geq 1}\frac{|u|^{2}}{|x|^{3}}\notag\\
&+ Ma_{2}^{2} + \frac{M(d-1)(d-3)a_{3}^{2}}{4}  - \C(M+1/2)a_{2}a_{3} \notag\\
& \leq \kappa\left[ (1+\lambda)|||u|||_{1}^{2} + |||\D u|||_{1}^{2} + a_{4}^{2}\right] + C(\varepsilon + 1)\left( N_{1}(f)\right)^{2}.\notag
\end{align}

\vspace{0.1cm}
Observe that we need 
\begin{equation}\notag
Ma_{2}^{2} + \frac{M(d-1)(d-3)}{4}a_{3}^{2} - \C(M+1/2)a_{2}a_{3} > 0,
\end{equation}
which is satisfied if
\begin{equation}\notag
\frac{1}{(d-1)(d-3)}{(\C)}^{2}\frac{(M+1/2)^{2}}{M^{2}} < 1.
\end{equation}
Letting $M \to \infty$, we obtain 
\begin{equation}\notag
{(\C)}^{2} < (d-1)(d-3),
\end{equation}
which is our assumption.

Consequently, noting that $|||\cdot ||| \geq |||\cdot |||_{1}$, taking $\kappa$, $\eta$ small enough and $\lambda_{1}=\lambda_{1}(\sigma,\kappa,j_{1}) > 0$ large enough, we conclude (\ref{landabig0}), which is our claim.
\end{proof}

The result is slightly different in the 3d-case.

\begin{thm}\label{3dlandagrande0}
For dimension $d=3$, let $\varepsilon >0$, $f$ such that $N_{1}(f)<\infty$ and assume that (\ref{condicionf0})-(\ref{(Q0)}) and
\begin{equation}\label{(b10')}
|B_{\tau}|\leq \frac{c}{|x|^{2-\alpha}} \quad \quad |x|\leq 1 \quad \quad \quad c, \alpha >0
\end{equation} 
hold. Then there exists $\lambda_{1}>0$ so that for any $\lambda \geq \lambda_{1}$ the solution $u\in H^{1}_{A}(\Rd)$ of the Helmholtz equation (\ref{2.1}) satisfies
\begin{align}
\lambda|||u|||_{1}^{2} + |||\D u|||_{1}^{2}& + \int \frac{|\D^{\bot}u|^{2}}{|x|} + \sup_{R>0}\frac{1}{R^{2}}\int_{|x|=R} |u|^{2}d\sigma_{R}\\
& \leq C(1+\varepsilon)(N_{1}(f))^{2},\notag
\end{align}
being $C$ independent of $\varepsilon$.
\end{thm}

\begin{proof}
The proof follows by the same method as in the $d>3$ case. We will use the same multipliers as in the previous theorem fixing $M=1/2$. The main difference is that when $d=3$ we do not get the term related to
$
\int \frac{|u|^{2}}{|x|^{3}}
$
on the left-hand side of the inequality. Therefore, it is not possible to estimate the magnetic term as in (\ref{magnetico}). This requires the assumption (\ref{(b10')}) on the magnetic field $B$. Thus in this case, using the same notation as in the previous theorem we obtain
\begin{align}
B_{1} & \leq \int_{|x|\leq r_{1}} |B_{\tau}||u||\D^{\bot}u| + \int_{r_{1}\leq |x|\leq 1} |B_{\tau}||\D^{\bot}u||u|\notag\\
&\leq \eta\left(\int_{|x|\leq r_{1}}\frac{|\D^{\bot}u|^{2}}{|x|} \right)^{1/2}\left(\sup_{R\leq r_{1}}\frac{1}{R^{2}}\int_{|x|=R} |u|^{2}\right)^{1/2}\notag\\
& +\frac{1}{4}\int_{r_{1}\leq |x|\leq 1} \frac{|\D^{\bot}u|^{2}}{|x|} + \sigma|||u|||_{1}^{2}.\notag
\end{align}

The rest of the proof runs as before.
\end{proof}

\begin{remark}\label{landabig1}
Note that if we did not take $\lambda$ big enough, we would obtain
\begin{align}
|||\D u|||_{1}^{2} + \int \frac{|\D^{\bot}u|^{2}}{|x|}  + &(d-3)\int \frac{|u|^{2}}{|x|^{3}} + \sup_{R>0} \frac{1}{R^{2}}\int_{|x|=R}|u|^{2}d\sigma_{R}\notag\\
& \leq C(1+\varepsilon)\{|||u|||_{1}^{2} + (N_{1}(f))^{2}\}.\notag
\end{align}
\end{remark}

\begin{remark}\label{d=1,2}
Since singularities on the potentials at the origin are allowed, we reduce to the case $d\geq 3$. When $d=1,2$, the problems come from the terms (\ref{phipsi}) and (\ref{epsilonterm1}). Similar results to those in \cite{PV1}, section 5 could be obtained for weaker singularities in dimension $d=2$. 
\end{remark}

\subsection{Sommerfeld radiation condition}

Our next goal is to quantify the Sommerfeld radiation condition proving that it is upper bounded by the Agmon-H\"ormander norm of the solution. To this end, the basic idea is to build the full form of the Sommerfeld terms, using the integral identities given in Appendix. We emphasize that since the Sommerfeld condition is applied at infinity, it is sufficient to know the behavior of the potentials for $|x|\geq R$, $R$ big enough.

\begin{pro}\label{propositionradiation}
For $d\geq 3$, let $\lambda_{0} > 0$, $\varepsilon > 0$, $f\in L^{2}_{\frac{1+\delta}{2}}(\Rd) \cap L^{2}_{\delta}(\Rd)$ and suppose that (\ref{condicionf0}) holds. Then, there exists a positive constant $C = C(\lambda_{0},\mu)$ such that for all $\lambda \geq \lambda_{0}$ the solution $u\in H^{1}_{A}(\Rd)$ of the equation (\ref{2.1}) satisfies
\begin{align}\label{sommerfeld0}
\int_{|x|\geq 1}& \left|\D u- i\lambda^{1/2}\frac{x}{|x|}u\right|^{2} \left( \frac{1}{(1+|x|)^{1-\delta}} + \varepsilon(1+|x|)^{\delta} \right)\\
& \leq C(1+\varepsilon) \left[ |||u|||_{1}^{2} +  (N_{1}(f))^{2}\right] + C\int_{|x|\geq 1} \left\{(1+|x|)^{1+\delta} + \varepsilon (1+|x|)^{2\delta} \right\} |f|^{2},\notag
\end{align}
for all $0<\delta <2$ such that $\delta < \mu$, where $\mu$ is given in Assumption \ref{ass1}.
\end{pro}

\begin{proof}
The proof consists in the construction of the squares of the left hand side of (\ref{sommerfeld0}). We use a combination of the identities of the lemmas \ref{appendix1} and \ref{appendix2}.  

Let us denote $r=|x|$ and we define a radial function $\Psi:\Rd \to \R$ by
\begin{equation}
\Psi(x)=\int_{0}^{|x|}\Psi'(s)ds,\notag
\end{equation}
with
\begin{equation}\label{del}
\Psi'(r)=(1+r)^{\delta}, \quad 0<\delta<2.
\end{equation}

Let us consider a cut off function $\theta \in C^{\infty}(\R)$ such that $0\leq \theta \leq 1$, $d\theta/dr \geq 0$ with 
\begin{displaymath}
\theta(r) = \left\{ \begin{array}{ll}
1 & \textrm{if $r \geq 2$}\\
0 & \textrm{if $r \leq 1$},
\end{array} \right.
\end{displaymath}
\noindent
and set $\theta(x)=\theta\left(|x|\right)$.

Let us compute
$$
(\ref{(4.3)}) + \frac{1}{2}(\ref{(4.11)}) + \lambda^{1/2}(\ref{(4.2)}) -\frac{\varepsilon}{2\lambda^{1/2}}(\ref{(4.11)})
$$
with the following choice of the multipliers
\begin{align}
&\nabla\psi(x) = \frac{x}{|x|} \Psi'(r)\theta(x)\notag\\
& \varphi(x) = \Psi''(r)\theta(x)\notag\\
& \varphi(x) = \Psi'(r)\theta(x)\notag\\
& \varphi(x) = \Psi'(r)\theta(x),\notag
\end{align}
respectively.

Note that by (\ref{del}) we have
\begin{equation}\notag
\frac{\Psi'}{r} - \frac{\Psi''}{2}> \frac{(2-\delta)}{2\delta}\Psi''.
\end{equation}
Thus since $0<\delta <2$, letting $\nu = \frac{2-\delta}{2\delta}>0$ and noting that
\begin{equation}\notag
\left|\D u - i \lambda^{1/2}\frac{x}{|x|}u \right|^{2} = |\D u|^{2} + \lambda|u|^{2} - 2\lambda^{1/2}\Im \frac{x}{|x|}\cdot \D u \bar{u},
\end{equation} 
we obtain
\begin{align}
&\frac{\delta}{2}\int (1+|x|)^{\delta-1}|\D^{r}u - i\lambda^{1/2}u|^{2}\theta +\delta\nu \int (1+|x|^{\delta-1})|\D^{\bot}u|^{2}\theta\label{ezkerra1}\\
&+\frac{1}{2}\int (1+|x|)^{\delta}\left( \theta'|\D^{r}u - i\lambda^{1/2}u|^{2} + \frac{\varepsilon}{\lambda^{1/2}}\left|\D u -i\lambda^{1/2}\frac{x}{|x|}u \right|^{2}\theta\right)\notag\\
& \leq \Re\int \nabla\left(\Psi'\theta' +\frac{(d-1)\Psi'\theta}{|x|} \right)\cdot\D u\bar{u} - \frac{\varepsilon\Re}{2\lambda^{1/2}}\int \nabla(\Psi'\theta)\cdot\D u\bar{u}\notag\\
& +\Im \int \Psi' B_{\tau}\cdot\D^{\bot}u\bar{u}\theta + \frac{1}{2}\int (\Psi'' V_{1} + \partial_{r}V_{1}\Psi')|u|^{2}\theta\notag\\
& +\frac{1}{2}\int \left(\frac{(d-1)\Psi'\theta}{|x|} + \Psi'\theta' \right)V_{2}|u|^{2} + \Re\int V_{2}\Psi'\D^{r}u\bar{u}\theta\notag\\
& -\Re\int f\Psi'\left\{\left[\theta \left( \frac{d-1}{2|x|} + \frac{\varepsilon}{2\lambda^{1/2}}\right) + \theta'\right]\bar{u} + (\D^{r}\bar{u} + i\lambda^{1/2}\bar{u})\theta\right\}\notag.
\end{align}

Let us now estimate the right hand-side of the above inequality applying similar arguments and using the same notation as in the proof of Theorem \ref{landagrande0}. Since
\begin{equation}\label{erreala}
\Re\D^{r} u\bar{u} = \Re (\D^{r}u-i\lambda^{1/2}u)\bar{u}
\end{equation}
and
$\delta <2$, the first term can be upper bounded
\begin{equation}\notag
\kappa\int |\D^{r}u -i \lambda^{1/2}u|^{2}\left((1+|x|)^{\delta-1}\theta + (1+|x|)^{\delta}\theta'\right) + C(\kappa)|||u|||_{1}^{2},
\end{equation}
for any $\kappa >0$. Concerning the $\varepsilon$ term, note that by integration by parts and the a priori estimate (\ref{a+++0}), we have
\begin{align}
-\frac{\varepsilon\Re}{2\lambda^{1/2}}\int \nabla(\Psi'\theta)\cdot \D u\bar{u} & = \frac{\varepsilon}{4\lambda^{1/2}}\int \Delta(\Psi'\theta)|u|^{2}\notag\\
& \leq \frac{C\varepsilon}{\lambda^{1/2}}\int_{|x|\geq 1} \frac{|u|^{2}}{(1+|x|)^{2-\delta}}\notag\\
& \leq \frac{C}{\lambda_{0}^{1/2}}N_{1}(f)|||u|||_{1}.\notag
\end{align}
We now pass to the terms containing the potentials. By (\ref{condicionf0}) it follows easily that for $\delta<\mu$ yields
\begin{align}
\frac{1}{2}\int \left[ (\Psi'' V_{1} + \partial_{r}V_{1}\Psi')\theta + \frac{(d-1)\Psi'\theta}{|x|}V_{2}+ \Psi'\theta'V_{2}\right]|u|^{2} \leq C|||u|||_{1}^{2}.\notag
\end{align}
If moreover, we apply the Cauchy-Schwarz inequality, then we get
\begin{equation}
\Im \int \Psi' B_{\tau}\cdot\D^{\bot}u\bar{u}\theta \leq C\left(\int |\D^{\bot}u|^{2}(1+|x|)^{\delta-1}\theta \right)^{1/2}|||u|||_{1}\notag
\end{equation}
and combining with (\ref{erreala}), gives
\begin{equation}
\Re\int V_{2}\Psi'\D^{r}u\bar{u}\theta \leq C\left(\int |\D^{r}u - i\lambda^{1/2}u|^{2}(1+|x|)^{\delta-1}\theta \right)^{1/2}|||u|||_{1}\notag.
\end{equation}
Thus the potential terms can be estimated by
\begin{equation}
\kappa \int (1+|x|)^{\delta-1} \left( |\D^{\bot}u|^{2} + |\D^{r}u-i\lambda^{1/2}u|^{2}\right)\theta + C|||u|||_{1}^{2}.\notag
\end{equation}
Finally, applying the same reasoning to the terms containing $f$, we obtain that they are upper bounded by
\begin{align}
& \kappa\int (1+|x|)^{\delta-1}|\D^{r}u -i\lambda^{1/2}u|^{2}\theta + C(\kappa)\int (1+|x|)^{1+\delta}|f|^{2}\theta \notag\\
& +C|||u|||_{1}\left(\int (1+|x|)^{1+\delta}|f|^{2}\theta \right)^{1/2}\notag\\
& +\frac{C}{\lambda^{1/2}}|||u|||_{1}^{1/2}(N_{1}(f))^{1/2}\left(\varepsilon\int (1+|x|)^{2\delta}|f|^{2}\theta \right)^{1/2}\notag.
\end{align}

As a consequence, choosing $\kappa$ small enough, we deduce (\ref{sommerfeld0}) and the proof is over.

\end{proof}

\begin{remark}
Recall from Remark \ref{r01} that we have been working under the condition that $r_{0}=1$ in Assumption \ref{ass1}. The same conclusion can be drawn for a general $r_{0}$. In this case, one should set $\theta_{r_{0}}(x) =\theta\left(\frac{|x|}{r_{0}} \right)$ and replace $\theta(x)$ by $\theta_{r_{0}}(x)$  in the above proof.
\end{remark}

\begin{remark}\label{remarkg(x)}
Observe that the previous proof does not work neither for the $\delta=0$ case, nor for the $\delta=2$ case. When $\delta=0$, $\Psi'(|x|)=1$ and $\Psi''(|x|)=0$. Then, we would not obtain the main square in the left hand side of the inequality. On the other hand, when $\delta=2$, the problem comes from the term 
\begin{equation}\notag
\Re\int \nabla\left(\Psi'\theta' +\frac{(d-1)\Psi'\theta}{|x|} \right)\cdot\D u\bar{u}.
\end{equation}
If $\Psi'(r) = (1+r)^{2}$ one needs to estimate the term $\int_{|x|\geq 1} \frac{|u|^{2}}{|x|}$, which is not upper bounded by $|||u|||_{1}^{2}$. Moreover, we do not get the estimate for the tangential component of the magnetic gradient and thus we are not able to absorb the term containing the magnetic field. The $\delta = 2$ case is particularly interesting and needs special attention so that it will be studied elsewhere. Both $\delta = 0$ and $\delta = 2$ cases have been studied in \cite{Zu}.
\end{remark}

\begin{remark}
Similarly, one could get the following version of the Sommerfeld radiation condition
\begin{align}\label{annulus}
\int_{C(j)}& \left| \D u - i\lambda^{1/2}\frac{x}{|x|}u \right|^{2} \frac{1}{(1+|x|)^{1-\delta}} \leq C(1+\varepsilon)\left[|||u|||_{1}^{2} + |||\D u|||_{1}^{2} + (N_{1}(f))^{2} \right]\\
& + C\int_{2^{j-2} \leq |x| \leq 2^{j+1}} \left[ (1+|x|)^{1+\delta} + \varepsilon (1+|x|)^{2\delta}\right] |f|^{2}\notag,
\end{align}
for any $j$ such that $j_{0} \leq j \leq j_{1}$ where $j_{0}, j_{1} > 0$ are fixed. This inequality will be very useful in what follows (see (\ref{conclusion0}) below). 

In order to get (\ref{annulus}), one only needs to define the cut-off function $\theta$ as
\begin{displaymath}
\theta(r) = \left\{ \begin{array}{ll}
0 & \textrm{if $r \geq 2$}\\
1 & \textrm{if $\frac{1}{2} \leq r \leq 1$}\\
0 & \textrm{if $r \leq \frac{1}{4}$}.
\end{array} \right.
\end{displaymath} 
and set $\theta_{j}(x) = \theta\left(\frac{|x|}{2^{j}} \right)$. Then we put $\theta_{j}$ instead of $\theta(x)$ in the definition of the multipliers above. The only difference is that in this case $\theta' \leq  0$ if $1 \leq r \leq 2$ so that one needs to estimate the term containing $\theta'$ in the left hand side of (\ref{ezkerra1}). The details are left to the reader.
\end{remark}

\subsection{Compactness argument when $\lambda \in [\lambda_{0}, \lambda_{1}]$}
Our next objective is to show that for any $\lambda \in [\lambda_{0}, \lambda_{1}]$, 
\begin{equation}\label{comp}
\lambda|||u|||_{1}^{2} \leq C (N_{1}(f))^{2}.
\end{equation}
In order to get this estimate, we begin by proving the following a priori estimate, which is a consequence of assumption (\ref{assumptionself}).

\begin{lem}\label{int}
For each $R>0$ any solution $u\in H^{1}_{A}(\Rd)$ of the equation (\ref{2.1}) satisfies
\begin{align}\label{elliptic1}
\int_{|x|\leq R} |\nabla_{A} u|^{2} \leq C(1+\lambda)\int_{|x|\leq R+1} |u|^{2} + \int_{|x|\leq R+1} |f|^{2}.
\end{align}
\end{lem}

\begin{proof}
Let $\psi \in C^{\infty}_{0}$ such that $0\leq \psi \leq 1$ and
\begin{equation}\notag
\psi(x) = \left\{ \begin{array}{ll}
1 & \textrm{if $|x| \leq R$},\\
0 & \textrm{if $|x| \geq R+1$}.
\end{array} \right.
\end{equation}
Note that any solution $u\in H^{1}_{A}(\Rd)$ of the equation (\ref{2.1}) satisfies
\begin{equation}\notag
(\nabla_{A}^{2} + V_{1} + V_{2} + \lambda + i\varepsilon) (\psi u) = \psi f + u\Delta\psi + 2\D u\cdot \nabla \psi.
\end{equation}
Let us multiply the above identity by $\psi\bar{u}$, integrate over $\Rd$ and take the real part. Hence, by integration by parts we get
\begin{align}
\int |\D(\psi u)|^{2} & \leq \lambda\int_{|x|\leq R+1} |u|^{2} + \int (V_{1} +V_{2})|\psi u|^{2} + \int_{|x|\leq R+1} |f||u|\notag\\
& + \int |\D(\psi u)||\nabla\psi||u| + \int |\psi||\Delta \psi| |u|^{2}.\notag
\end{align}
Now by the assumption (\ref{assumptionself}) on the potentials $V_{1}, V_{2}$ and the diamagnetic inequality (\ref{diamagnetic}) we have
\begin{align}\notag
\int (V_{1} + V_{2})|\psi u|^{2} &< \int |\D(\psi u)|^{2}.
\end{align}
Hence, by Cauchy-Schwarz inequality it follows that
\begin{align}\label{ellipticgradient}
\int |\D (\psi u)|^{2} \leq C(1+\lambda)\int_{|x|\leq R+1} |u|^{2} + \int_{|x|\leq R+1} |f|^{2},
\end{align}
which gives (\ref{elliptic1}) and the lemma follows.

\end{proof}

\begin{remark}
Note that since
\begin{align}\notag
\int |\nabla (\psi u)|^{2} \leq C\int (|\D (\psi u)|^{2} + |A\psi u|^{2}),
\end{align}
applying the condition (\ref{extracondition}) on the magnetic potential $A$ to $|u|$, then by the diamagnetic inequality (\ref{diamagnetic}), it follows that
\begin{equation}\notag
\int |\nabla (\psi u)|^{2} \leq C\int |\D (\psi u)|^{2}.
\end{equation}
This combined with (\ref{ellipticgradient}) gives the well known elliptic a priori estimate
\begin{align}\notag
\int_{|x|\leq R} |\nabla u|^{2} \leq C(1+\lambda)\int_{|x|\leq R+1} |u|^{2} + \int_{|x|\leq R+1} |f|^{2}
\end{align}
for solutions of the equation (\ref{2.1}).
\end{remark}

Having disposed of this preliminary step, we can return to show (\ref{comp}).

\begin{pro}\label{proposition}
For $d\geq 3$, let $\lambda_{0} > 0$, $\lambda \in [\lambda_{0}, \lambda_{1}]$, with $\lambda_{1} >\lambda_{0}$ and $\varepsilon \in (0, \varepsilon_{1})$. Under the assumptions of Proposition \ref{propositionradiation} above, if moreover (\ref{extracondition}) holds, then the solution of the equation (\ref{2.1}) satisfies
\begin{equation}\label{landapeque–o0}
\lambda|||u|||_{1}^{2} + |||\D u|||_{1}^{2}  \leq C (1+\varepsilon) (N_{1}(f))^{2},
\end{equation}
where $C$ is independent of $\varepsilon$.
\end{pro}

\begin{proof}
Our proof starts recalling that
\begin{equation}\notag
\left|\D u - i\lambda^{1/2}\frac{x}{|x|}u\right|^{2} = |\D u|^{2} + \lambda|u|^{2} -2\Im \lambda^{1/2}\frac{x}{|x|}\cdot \D u \bar{u}.
\end{equation}
Let us integrate the above identity over the sphere $S_{r}:=\{|x|=r\}$, obtaining
\begin{align}\label{3.7}
\int_{S_{r}} (\lambda|u|^{2} + |\D u|^{2}) d\sigma_{r} &= \int_{S_{r}} \left|\D u-i\lambda^{1/2}\frac{x}{|x|}u\right|^{2} d\sigma_{r}+ 2\Im\lambda^{1/2}\int_{S_{r}}\D^{r} u \bar{u} d\sigma_{r}.
\end{align}

Let us multiply now equation $(\ref{2.1})$ by $\bar{u}$, integrate it over the ball $B_{r}:=\{|x|\leq r\}$ and take the imaginary part. Since $\varepsilon >0$, it follows that
\begin{equation}\notag
\Im\int_{S_{r}} \D^{r} u\bar{u} d\sigma_{r} \leq \Im\int_{B_{r}}f\bar{u}.
\end{equation}
Combining this with (\ref{3.7}) yields
\begin{equation}\label{3.71}
\int_{S_{r}} (\lambda|u|^{2} + |\D u|^{2}) d\sigma_{r} \leq \int_{S_{r}} \left|\D u-i\lambda^{1/2}\frac{x}{|x|}u\right|^{2} d\sigma_{r} +2\Im\lambda^{1/2}\int_{B_{r}}f\bar{u}.
\end{equation}

Now, let $R > \rho \geq 1$ and denote $j_{0}$ and $j_{1}$ by $2^{j_{0}-1} < \rho < 2^{j_{0}}$ and $2^{j_{1}-1} < R < 2^{j_{1}}$, respectively. Let us multiply both sides of (\ref{3.71}) by $\frac{1}{R}$ and integrate from $\rho$ to $R$ with respect to $r$. Then we have
\begin{align}\label{concl}
\frac{1}{R}\int_{\rho \leq |x| \leq R} (\lambda|u|^{2} + |\D u|^{2})  & \leq \frac{1}{R}\sum_{j=j_{0}}^{j_{1}} \int_{C(j)}\left|\D u -i \lambda^{1/2}\frac{x}{|x|}u\right|^{2}\\
& + \kappa\lambda|||u|||_{1}^{2} + C(\kappa)(N_{1}(f))^{2}\notag\\
& \equiv I_{1} + I_{2},\notag
\end{align}
for $\kappa > 0$ and by (\ref{annulus}) we get 
\begin{align}\label{conclusion0}
I_{1} & \leq \frac{1}{R}\sum_{j=j_{0}}^{j_{1}}(1+2^{j})^{1-\delta}\int_{C(j)}\frac{1}{(1+2^{j})^{1-\delta}} \left|\D u -i \lambda^{1/2}\frac{x}{|x|}u\right|^{2} \\
& \leq C(1+\varepsilon)\sum_{j=j_{0}}^{j_{1}} \frac{(1+2^{j})^{1-\delta}}{2^{j_{1}}}(|||u|||_{1}^{2} +(N_{1}(f))^{2})\notag\\
& + C(1+\varepsilon)\sum_{j=j_{0}}^{j_{1}} \frac{(1+2^{j})^{1-\delta}}{2^{j_{1}}} \int_{2^{j}\geq \frac{1}{2}} \left[ (1+2^{j})^{1+\delta} + \varepsilon(1+ 2^{j})^{2\delta} \right] |f|^{2} \notag\\
& \leq C(1+\varepsilon)\left[ \sum_{j=j_{0}}^{j_{1}} 2^{-\delta j}|||u|||_{1}^{2} + \left(1+\sum_{j=j_{0}}^{j_{1}} \frac{(1+2^{j}) + \varepsilon (1+2^{j})^\delta}{2^{j_{1}}} \right)(N_{1}(f))^{2}\right]. \notag
\end{align}
As a consequence, from (\ref{concl}) and (\ref{conclusion0}), taking $\kappa$ small enough and $\rho$ big enough, we deduce
\begin{equation}\notag
\frac{1}{R}\int_{\rho \leq |x| \leq R} (\lambda|u|^{2} + |\D u|^{2})  \leq \frac{\lambda}{2} |||u|||_{1}^{2} + C(1+\varepsilon)(N_{1}(f))^{2}.
\end{equation}

It remains to prove that
\begin{equation}\label{contradiccion}
\int_{|x|\leq \rho} (\lambda|u|^{2} + |\D u|^{2})  \leq C(N_{1}(f))^{2}.
\end{equation}
Let us assume that (\ref{contradiccion}) is false. Then, for each $n \in \mathbb{N}$, there exist $\varepsilon_{n} \in (0, \varepsilon_{1})$ with $0<\varepsilon_{1} < \infty$, $\lambda_{n}\in [\lambda_{0}, \lambda_{1}]$ and $u_{n}, f_{n}$ such that
\begin{equation}\notag
(\nabla + iA)^{2}u_{n} + (V_{1} + V_{2})u_{n} + \lambda_{n}u_{n} + i\varepsilon_{n}u_{n} = f_{n},
\end{equation}
with
\begin{equation}\label{A}
\int_{|x| \leq \rho} (\lambda_{n}|u_{n}|^{2} + |\D u_{n}|^{2}) =1
\end{equation} 
and
\begin{equation}\label{B}
N_{1}(f_{n}) \leq \frac{1}{n} \quad \quad \left(\lim_{n \to \infty} N_{1}(f_{n}) = 0\right).
\end{equation}
Since $\lambda_{n} \in [\lambda_{0}, \lambda_{1}]$ and $\varepsilon_{n} \in (0 , \varepsilon_{1})$, we may assume with no loss of generality that $\lambda_{n} \to \lambda^{0}$ and $\varepsilon_{n} \to \varepsilon^{0}$ where $\lambda^{0} \in [\lambda_{0}, \lambda_{1}]$, $\varepsilon^{0} \in [0, \varepsilon_{1}]$, as $n$ tends to $\infty$.

On the other hand, from (\ref{A}) and condition (\ref{extracondition}) on $A$, one can easily deduce that $\{u_{n}\}$ is a bounded sequence in $H^{1}_{loc}(\Rd)$. Hence, by the Rellich-Kondrachov theorem, one can conclude that there exists a subsequence of $u_{n}$, $u_{n_{p}}$, such that
\begin{equation}\notag
u_{n_{p}} \to u \quad \textrm{in} \quad L^{2}_{loc}(\Rd), \quad \textrm{as} \quad p \to \infty, \quad \textrm{with} \quad u\in L^{2}_{loc}(\Rd),
\end{equation}
which implies
\begin{equation}\notag
\sup_{R>1}\frac{1}{R}\int_{|x|\leq R} |u_{n_{p}} - u|^{2}dx \quad \quad \to \quad \quad 0
\end{equation}
and by (\ref{B}) it follows that
\begin{equation}\notag
((L+\lambda^{0}+i\varepsilon^{0})u, \varphi) = (0, \varphi) \quad \forall \varphi \in C_{0}^{\infty}.
\end{equation}
Moreover, if we denote $v_{n} = u_{n_{p}}-u$, since
\begin{align}
g_{p}&\equiv (L+\lambda^{0}+i\varepsilon^{0})v_{n}\notag\\
&= i(\varepsilon^{0} -\varepsilon_{n_{p}})u_{n_{p}} + (\lambda^{0}-\lambda_{n_{p}})u_{n_{p}} + f_{n_{p}} - (L+\lambda^{0} + i\varepsilon^{0})u,\notag
\end{align}
applying Lemma \ref{int} to $v_{n}$ and $g_{p}$, one can deduce for $R>0$
\begin{align}
\int_{|x|\leq R} |\D v_{n}|^{2} & \leq C(1+\lambda)\int_{|x|\leq R+1} |v_{n}|^{2} + \int_{|x|\leq R+1} |g_{p}|^{2}.\notag
\end{align}
Hence,
\begin{equation}\label{3.73}
\D u_{n_{p}}  \to \D u \quad \textrm{in} \quad L^{2}_{loc}(\Rd), \quad \textrm{as} \quad p \to \infty, \quad \textrm{with} \quad \D u\in L^{2}_{loc}(\Rd).
\end{equation}
As a consequence, by (\ref{A}) $u$ satisfies
\begin{equation}\label{contra}
\int_{|x|\leq \rho} (\lambda|u|^{2} + |\D u|^{2}) =1
\end{equation}
and
\begin{equation}\label{0010}
(\nabla + iA)^{2}u + (V_{1} + V_{2})u + \lambda^{0} u + i\varepsilon^{0}u = 0
\end{equation}
in the distributional sense. Thus by uniqueness of solution of the equation (\ref{0010}), we conclude that $u\equiv 0$, which contradicts (\ref{contra}).

We have thus proved that for $R > 1$
\begin{equation}\notag
\frac{1}{R}\int_{|x|\leq R} (\lambda|u|^{2} + |\D u|^{2})  \leq \frac{\lambda}{2}|||u|||_{1}^{2} + C(1+\varepsilon)(N_{1}(f))^{2}.
\end{equation}
Taking the supremum over $R\geq 1$, we get (\ref{landapeque–o0}) and the proof is complete.

\end{proof}

\subsection{Limiting absorption principle}\label{LAPIS}
Our next concern will be the existence of solution of the equation (\ref{res}), which is stated in the following lemma.
\begin{lem}\label{Lemma1.11}
Let $\lambda >0$, $\{u_{n}\}$ be a sequence such that for any $\rho>0$
\begin{equation}\label{somsom}
\int_{|x|\leq\rho} (\lambda|u_{n}|^{2} + |\D u_{n}|^{2})  < +\infty
\end{equation}
and let $\varepsilon_{n} \in (0,1)$ be a convergent sequence with $\varepsilon_{n} \to 0$ as $n\to \infty$, $f$ such that $N_{1}(f)<\infty$. Assume that
\begin{align}
&(L + \lambda + i\varepsilon_{n})u_{n} = f\notag
\end{align}
and $\{u_{n}\}$ satisfies the radiation condition
\begin{equation}
\int_{|x|\geq 1} \left|\D u_{n} -i\lambda^{1/2}\frac{x}{|x|}u_{n} \right|^{2}(1+|x|)^{\delta-1} < + \infty
\end{equation}
for some $\delta >0$ and for all $n=1, 2, \ldots$. Under the assumptions of Theorem \ref{unicidad1}, if moreover (\ref{extracondition}) holds, then $\{u_{n}\}$ has a strong limit $u$ in $(H^{1}_{A})_{loc}(\Rd)$ such that satisfies
\begin{align}
& (L + \lambda)u = f\notag\\
& \int_{|x|\leq \rho} (\lambda|u|^{2} + |\D u|^{2})  < +\infty\notag\\
&\int_{|x|\geq 1} \left|\D u -i\lambda^{1/2}\frac{x}{|x|}u \right|^{2}(1+|x|)^{\delta-1} < + \infty,\notag
\end{align}
for $\delta >0$.
\end{lem}

\begin{proof}
This follows by the compactness argument, in much the same way as in the proof of Proposition \ref{proposition}. Since $\varepsilon_{n} \to 0$ as $n\to \infty$, the same reasoning applies to this case and we deduce that there exists a subsequence of $u_{n}$, $u_{n_{p}}$, such that $u_{n_{p}} \to u$ in $(H^{1}_{A})_{loc}(\Rd)$ as $p\to \infty$ where $u\in (H^{1}_{A})_{loc}(\Rd)$ and satisfies
\begin{align}
& (\nabla + iA)^{2} u + (V_{1} + V_{2})u + \lambda u = f,\notag\\
& \int_{|x|\leq \rho} (\lambda|u|^{2} + |\D u|^{2}) < \infty.\notag
\end{align}

In addition, if we denote $\mathcal{D}u = \D u - i\lambda^{1/2}\frac{x}{|x|}u$, we also get that $\mathcal{D} u_{n_{p}}$ converges to $\mathcal{D}u$ in $L^{2}_{loc}(E_{1})$, where $E_{1}=\{|x|\geq 1\}$. As a consequence, we obtain $\mathcal{D}u_{n_{p}} \to \mathcal{D}u$ in $L^{2}_{\frac{\delta-1}{2}}(E_{1})$ satisfying $\int_{|x|\geq 1} \left|\D u - i\lambda^{1/2}\frac{x}{|x|}u \right|^{2}(1+|x|)^{\delta-1} < \infty.$

Finally, we shall show that the sequence $\{u_{n}\}$ itself converges in $(H^{1}_{A})_{loc}(\Rd)$ to the $u$ obtained above, which in turn implies that $\{\mathcal{D}u_{n}\}$ converges to $\{\mathcal{D}u\}$ in $L^{2}_{loc}(E_{1})$.  In fact, let us assume that there exists a subsequence $\{n_{q}\}$ of $\{n\}$ such that
\begin{equation}\label{u}
\Vert u- u_{n_{q}} \Vert_{L^{2}_{loc}} + \Vert \D u- \D  u_{n_{q}} \Vert_{L^{2}_{loc}}  \geq \gamma \quad (q=1,2,\ldots)
\end{equation}
with some $\gamma > 0$. Then, proceeding as above, we can find a subsequence $\{n^{'}_{q}\}$ of $\{n_{q}\}$ which satisfies
\begin{equation}\label{un}
u_{n^{'}_{q}} \to u^{'} \quad \textrm{in} \quad (H^{1}_{A})_{loc}(\Rd),
\end{equation}
$u^{'}$ being a solution in $(H^{1}_{A})_{loc}(\Rd)$ of $\D^{2}u^{'} + \lambda u^{'} + (V_{1} + V_{2})u^{'}  = f$
such that $\int_{|x|\geq 1} \left|\D u' -i\lambda^{1/2}\frac{x}{|x|}u' \right|^{2}(1+|x|)^{\delta-1} < + \infty$. Finally, by Theorem \ref{unicidad1} we assert that $u'$ obtained above is unique which implies that $u$ and $u'$ must coincide. Hence, from (\ref{un}) it follows that $u_{n_{q}} \to u \quad \textrm{in} \quad (H^{1}_{A})_{loc}(\Rd),$ which contradicts (\ref{u}). Thus $\{u_{n}\}$ converges to $u$ in $(H^{1}_{A})_{loc}(\Rd)$ and the lemma follows.
\end{proof}

Finally, the preceding lemma together with the uniqueness result for (\ref{res}) (Theorem \ref{unicidad1}) allows us to construct the unique solution $u=u(\lambda, f)$ as the limit of a sequence of solutions $\{u_{n}= u(\lambda+i\varepsilon_{n}, f)\}$ ($\varepsilon_{n} \to 0$) obtained above.

\begin{thm}[Limiting absorption principle]\label{LAP}
Under the hypotheses of Theorem \ref{d>30}, let $\{\varepsilon_{n}\} \subset (0,1)$ be a sequence tending to $0$. Let $u_{n}=u(\lambda+i\varepsilon_{n}, f)$. Then $\{u_{n}\}$ converges in $(H^{1}_{A})_{loc}(\Rd)$ to a $u$ such that
\begin{equation}
\lambda|||u|||_{1}^{2} + |||\D u|||_{1}^{2} \leq C(N_{1}(f))^{2},
\end{equation}
where $C=C(\lambda_{0})>0$ and solves $(L+\lambda)u=f.$

The limit $u=u(\lambda, f)$ is independent of the choice of the sequence $\{\varepsilon_{n}\}$ and is determined as the unique solution of the equation (\ref{res}) that satisfies the Sommerfeld radiation condition
$$
\int_{|x|\geq 1}(1+|x|)^{\delta-1} \left|\D u - i\lambda^{1/2}\frac{x}{|x|}u \right|^{2} \leq C\int (1+|x|)^{1+\delta}|f|^{2}, 
$$
for any $0<\delta < 2$, being $C=C(\lambda_{0})>0$.
\end{thm}

\begin{proof}
Let $f\in L^{2}_{\frac{1+\delta}{2}}(\Rd)\cap L^{2}_{\delta}(\Rd)$. Take $\{\varepsilon_{n}\} \subset (0,1)$ such that $\varepsilon_{n} \to 0$ as $n\to \infty$. We know that there exists a unique solution $u_{n}\in H^{1}_{A}(\Rd)$ of the equation $(L+\lambda+i\varepsilon_{n})u_{n} = f$ satisfying
\begin{align}
& \lambda|||u_{n}|||_{1}^{2} + |||\D u_{n}|||_{1}^{2} \leq C(\varepsilon_{n} + 1)(N_{1}(f))^{2}\notag\\
& \Vert \mathcal{D} u_{n}\Vert_{L^{2}_{\frac{\delta-1}{2}}(E_{1})} \leq C\left[ (1+\varepsilon_{n})\Vert f\Vert_{\frac{1+\delta}{2}}^{2} + \varepsilon_{n} \Vert f\Vert_{\delta}^{2} \right]\notag
\end{align}
for all $n=1, 2,\ldots$ where $\mathcal{D} u = \D u - i\lambda^{1/2}\frac{x}{|x|}u$ and $E_{1}=\{|x|\geq 1\}$. Then one can see from Lemma \ref{Lemma1.11} that $\{u_{n}\}$ has a strong limit in $(H^{1}_{A})_{loc}(\Rd)$ which is a solution of the equation $(L+\lambda)u=f$ and it is easy to check that satisfies
\begin{align}
& \lambda|||u|||_{1}^{2} + |||\D u|||_{1}^{2} \leq C(N_{1}(f))^{2}\notag\\
& \Vert \mathcal{D} u\Vert_{L^{2}_{\frac{\delta-1}{2}}(E_{1})} \leq C\Vert f\Vert_{\frac{1+\delta}{2}}^{2}.\label{somerradiazio}
\end{align}
By the uniqueness result (see Theorem \ref{unicidad1}), it follows that the $u$ obtained above is a unique solution of $(L+\lambda)u=f$ satisfying $(\ref{somerradiazio})$ and the proof is complete.
\end{proof}

\section{Proof of Theorem \ref{unicidad1}}\label{unic}
The proof is based on multiplier method and integration by parts. It will be divided into three steps.

Let $R>2r_{0}\geq 1$, $r_{0}$ being as in Assumption \ref{ass1}, which can be taken as $r_{0}=1$. Our first goal is to show that there exists $\mu>0$ such that
\begin{equation}\label{1.100}
\int_{|x| > R} (|\D u|^{2} + |u|^{2}) \leq \frac{C}{R^{1+\mu}} \int_{\frac{R}{2} \leq |x| \leq R} |u|^{2}.
\end{equation}
For this purpose, we multiply the equation (\ref{homogeneo}) by 
\begin{equation}\label{symanti}
\nabla\psi \cdot \overline{\D u} + \frac{1}{2} \Delta\psi\bar{u} + \varphi\bar{u},
\end{equation}
where $\psi, \varphi$ are regular radial real-valued functions, and we integrate it over the ball $\{|x|<R_{1}\}$ with $R_{1}>R$, obtaining 
\begin{align}\label{(4.310)}
&\int_{|x| < R_{1}} \D u\cdot D^{2}\psi \cdot \overline{\D u} -\int_{|x|<R_{1}}\varphi|\D u|^{2} + \int_{|x|<R_{1}} \varphi \lambda |u|^{2}\\
& =\frac{1}{4}\int_{|x| < R_{1}}( \Delta^{2}\psi -2\Delta\varphi) |u|^{2} - \int_{|x|<R_{1}}\varphi V_{1}|u|^{2} - \int_{|x|<R_{1}}\varphi V_{2}|u|^{2}\notag\\
& - \frac{1}{2}\int_{|x| < R_{1}} \psi' \partial_{r} V_{1} |u|^{2} + \Im \int_{|x| < R_{1}} \psi' B_{\tau}\cdot \D u \bar{u}\notag\\
& + \frac{1}{2}\int_{|x| < R_{1}} V_{2}\Delta\psi|u|^{2} + \Re \int_{|x|<R_{1}}V_{2}\nabla\psi\cdot \D u\bar{u} -\Re\int_{|x|=R_{1}} \D^{r} u\varphi\bar{u} \notag\\
& + \frac{1}{4} \int_{|x|=R_{1}} \left( \nabla(\Delta\psi) -2\nabla\varphi\right) \cdot\frac{x}{|x|} |u|^{2}  + \frac{1}{2}\Re \int_{|x|=R_{1}}  \overline{\D^{r} u} \Delta\psi u  \notag\\
& - \frac{1}{2}\int_{|x|=R_{1}} \frac{x}{|x|} \cdot\nabla\psi |\D u|^{2} + \frac{1}{2} \int_{|x|=R_{1}} (\lambda +V_{1})\frac{x}{|x|}\cdot\nabla\psi|u|^{2}.\notag
\end{align}
Let us consider a cut off function $\theta$ with 
\begin{displaymath}
\theta(r) = \left\{ \begin{array}{ll}
1 & \textrm{if $r \geq 1$}\\
0 & \textrm{if $r < \frac{1}{2},$}
\end{array} \right.
\end{displaymath}
\noindent
$\theta' \geq 0$ for all $r$, and set  $\theta_{R}(x) = \theta \left(\frac{|x|}{R} \right)$. Then, for $R$ such that $\frac{R}{2} >r_{0}\geq1$ and $R<R_{1}$ we define the multiplier $\psi$ such that
\begin{equation}\label{psi0}
\nabla\psi(x)= \frac{x}{R}\theta_{R}(x)
\end{equation}
and $\varphi$ by
\begin{equation}\label{phi0}
\varphi(x)=\frac{1}{2R}\theta_{R}(x).
\end{equation}
Let us insert (\ref{psi0}) and (\ref{phi0}) into the identity (\ref{(4.310)}). Hence, by (\ref{hessianoa}) the left-hand side can be lower bounded by
\begin{align}\label{ezkerra}
 \int_{|x| < R_{1}} &\D u\cdot D^{2}\psi \cdot \overline{\D u} -\int_{|x|<R_{1}}\varphi|\D u|^{2} +\int_{|x|<R_{1}} \varphi \lambda |u|^{2} \\
&> \frac{1}{2R}\int_{|x|<R_{1}} \left(|\D u|^{2} + \lambda |u|^{2}\right)\theta_{R}.\notag
\end{align}

Regarding to the right-hand side of (\ref{(4.310)}), first note that
\begin{equation}
\frac{1}{4}\int_{|x| < R_{1}} (\Delta^{2}\psi - 2\Delta\varphi) |u|^{2} \leq \frac{C}{R^{3}}\int_{\frac{R}{2}<|x|<R} |u|^{2}.\notag
\end{equation}
In order to analyze the terms containing the potentials, here and subsequently, we will use $\eta=\eta(R)$ to denote a positive constant depending on $R$ that tends to $0$ as $R$ tends to infinity. Thus by (\ref{condicionf0}) and the Cauchy-Schwarz inequality we have
\begin{align}
\Im \int_{|x|<R_{1}} \psi' B_{\tau} \cdot \D u\bar{u} & \leq \int_{|x|<R_{1}} \frac{|B_{\tau}||x|}{R}|u||\D u|\notag\\
& \leq \sum_{j=j_{1}}^{j_{2}}2^{-j\mu}\int_{|x|<R_{1}}\theta_{R}|u||\D u|\notag\\
& \leq \eta(R)\int_{|x|<R_{1}} (|u|^{2} + |\D u|^{2})\theta_{R}.\notag
\end{align}
Similarly,
\begin{align}
\Re \int_{|x|<R_{1}} V_{2}\nabla\psi\cdot \D u\bar{u} & \leq \eta(R) \int_{|x|<R_{1}} (|u|^{2}+|\D u|^{2})\theta_{R},\notag
\end{align}
\begin{align}
-\int_{|x|<R_{1}} \left(\frac{\psi' \partial_{r}V_{1}}{2} +\varphi V_{1} \right)|u|^{2} & \leq \eta(R)\int_{|x|<R_{1}} |u|^{2}\theta_{R}.\notag
\end{align}
Finally, since $supp\, \theta'_{R} \subset \{\frac{R}{2} < |x| < R\}$, yields
\begin{align}
\int_{|x|<R_{1}} \left(\frac{\Delta\psi}{2} -\varphi \right) V_{2}|u|^{2} &  \leq \eta(R) \int_{|x|<R_{1}}|u|^{2}\theta_{R}\notag\\
&  + \frac{c}{2R^{2+\mu}} \int_{\frac{R}{2} <|x|<R} |u|^{2}.\notag
\end{align}
Let us analyze now the surface integrals of the equality (\ref{(4.310)}). An easy computation shows that by (\ref{psi0}), (\ref{phi0}) and condition (\ref{condicionf0}) applying to $V_{1}$, the boundary terms are upper bounded by
\begin{align}\label{esferas}
\frac{1}{R}\int_{|x|=R_{1}} |u||\D^{r}u| +\frac{1}{2}\int_{|x|=R_{1}} (|\D u|^{2} + \lambda|u|^{2})  +\frac{1}{2R_{1}^{\mu}} \int_{|x|=R_{1}}|u|^{2}.
\end{align}
As a consequence, from (\ref{ezkerra})-(\ref{esferas}) yields
\begin{align}
\frac{1}{2R}\int_{|x|<R_{1}}(|\D u|^{2} + \lambda|u|^{2})\theta_{R}  & \leq  \eta(R)\int_{|x|<R_{1}} (|u|^{2} + |\D u|^{2})\theta_{R}\notag\\
& + \frac{C}{R^{2+\mu}}\int_{\frac{R}{2}<|x|<R}|u|^{2}\notag\\
& + C(\lambda_{0})\int_{|x|=R_{1}} \left\{ |\D u|^{2} + \lambda |u|^{2})\right\} .\notag
\end{align}

Now, taking $R$ large enough such that
\begin{equation}\notag
\frac{\min(1,\lambda)}{2} - \eta(R)  > 0,
\end{equation}
it follows that
\begin{align}
\frac{1}{R}\int_{R<|x|<R_{1}} (|\D u|^{2}+|u|^{2})  & \leq \frac{C}{R^{2+\mu}} \int_{\frac{R}{2}<|x|<R} |u|^{2} + C\int_{S_{R_{1}}}(|\D u|^{2}+ \lambda|u|^{2}).\notag
\end{align}
Letting $R_{1}\to \infty$ in the above inequality, by (\ref{1.130}) we get (\ref{1.100}), which is our claim.
\bigskip

Our next step is to prove that for $R > 2r_{0} \geq 1$ and any $m\geq 0$, then
\begin{equation}\label{0.6}
\int_{|x| > R} |x|^{m} (|\D u|^{2} + |u|^{2}) < +\infty.
\end{equation}
We do it by induction. Let $\gamma = 1+\mu$ and first note that from the first step one can easily deduce that for any $R\geq 1$ holds
\begin{align}
\int_{|x|\geq 2R} |x|^{\gamma} &(|u|^{2}+|\D u|^{2}) \leq \sum_{j\geq J}(2^{j\gamma}) \int_{2^{j-1}\leq |x|\leq 2^{j}} (|u|^{2}+|\D u|^{2})\notag\\
& \leq C \sum_{j\geq J} \int_{2^{j-2}\leq |x|\leq 2^{j-1}} |u|^{2}\leq C\int_{|x|\geq R} (|u|^{2} + |\D u|^{2})\notag\\
& \leq \frac{C}{R^{\gamma}} \int_{\frac{R}{2}\leq |x|\leq R}|u|^{2},\notag
\end{align}
being $J$ such that $2^{J-1}\leq 2R \leq 2^{J}$. The same conclusion can be drawn for any $m\geq 0$. Indeed, assuming that
\begin{equation}\label{mntzat}
\int_{|x|\geq R} |x|^{m}(|u|^{2} + |\D u|^{2}) \leq \frac{C}{R^{1+\mu}}\int_{\frac{R}{2}\leq |x| \leq R} |u|^{2},
\end{equation}
it follows that (\ref{mntzat}) is true when $m$ is replaced by $m+\gamma$. Thus we obtain (\ref{0.6}).

We next claim the exponential decay. Let us multiply again the equation (\ref{homogeneo}) by (\ref{symanti}), but instead of integrating over a ball, we do it over the whole $\Rd$. Note that this is equivalent to adding the identities (\ref{(4.11)}) and (\ref{(4.3)}) with $f=0$. Thus we get the identity (\ref{batuketa}) with the right-hand side equals to $0$. Let us now choose the multipliers as
$$\nabla\psi(x)= |x|^{m+1}\frac{x}{|x|}\theta_{R}(x),$$
$$\varphi(x) = \frac{1}{2}|x|^{m}\theta_{R}(x),$$
for $R\geq 2r_{0}\geq 1$, $m\geq 1$ and $\theta_{R}$ being as above. 
\\
For simplicity of notation, we continue to write $\eta=\eta(R)$ for a function depending on $R$ such that $\eta(R) \to 0$ as $R\to \infty$. Thus analysis similar to that in the first step shows that taking $R$ large enough such that
\begin{equation}\notag
\frac{\min\{1,\lambda\}}{2} - \eta(R)>0,
\end{equation}
we get
\begin{align}
\int |x|^{m} (|\D u|^{2} + |u|^{2})\theta_{R} & \leq  \int \left( \eta(R)m|x|^{m-1} + Cm^{3}|x|^{m-2}\right)|u|^{2}\theta_{R}\notag\\
& + \left(\frac{Cm^{2}}{R^{2}} + \frac{c}{2R^{1+\mu}}\right)\int_{\frac{R}{2}<|x|<R}|x|^{m}|u|^{2}.\notag
\end{align}
Let us take now $m = \delta l$ with $0<\delta<2/3$ and multiply both sides of the above inequality by $\frac{t^{l}}{l!}$, $t \geq 1$ and $l \geq 3$. Making the sum with respect to l from $3$ to $\infty$ we have
\begin{align}
&\left(1- \frac{2t}{3}R^{\delta-1}\eta(R)-\frac{9}{2}R^{3\delta-2}t^{3}\right)\int e^{|x|^{\delta}t}(|\D u|^{2} + |u|^{2})\theta_{R}\notag \\ 
&  \leq \int (|\D u|^{2}+|u|^{2})\left(1+t|x|^{\delta} + \frac{t^{2}}{2}|x|^{2\delta} \right)\theta_{R}\notag\\
& + \left(CR^{2(\delta-1)}t^{2} + \frac{c}{2R^{1+\mu}}\right)\int_{\frac{R}{2}<|x|<R}e^{|x|^{\delta}t}|u|^{2}.\notag
\end{align}
Fix $t\geq 1$ and $0<\delta<\frac{2}{3}$. Then, for sufficiently large $R=R(t)$ such that
$$
\frac{2t}{3}R^{\delta-1}\eta(R) + \frac{9}{2}t^{3}R^{3\delta-2} < 1,
$$
by (\ref{0.6}) we conclude that 
\begin{equation}\notag
\int_{|x|>R} e^{|x|^{\delta}t} (|\D u|^{2}+|u|^{2}) <+\infty.
\end{equation}
Therefore,
\begin{equation}\notag
\int e^{|x|^{\delta}t} (|\D u|^{2}+|u|^{2}) <+\infty
\end{equation}

We are now in a position to show that $u=0$ almost everywhere in $\{ |x|>2R \}$. Set $v=e^{t|x|^{\delta}/2}u$ with $t \geq 1$ and $0< \delta < 2/3$. Then, by a direct computation $v$ satisfies the equation
\begin{align}\label{3.9}
&\D^{2} v +[\lambda+V_{1} + V_{2}]v -\delta t|x|^{\delta-1}\frac{x}{|x|}\cdot \D v\\
& +\left[\frac{\delta^{2}t^{2}|x|^{2(\delta-1)}}{4}-\frac{\delta(\delta +d-2)t|x|^{\delta-2}}{2}\right]v=0.\notag
\end{align}
We multiply (\ref{3.9}) by 
$$
|x|\frac{x}{|x|}\cdot\overline{\D v}+\frac{d-1}{2}\bar{v}
$$ 
(the combination of the symmetric and the antisymmetric multipliers, (\ref{symanti}) with \hbox{$\nabla\psi=x$} and \hbox{$\varphi=-1/2$}), integrate it over $\{|x|>R\}$ for some $R>2r_{0}$ and take the real part. Hence, it follows that
\begin{align}
& \frac{\min\{1,\lambda\}}{2}\int_{|x|>R}(|\D v|^{2} +|v|^{2}) +\frac{(2\delta-1)\delta^{2}t^{2}}{4}\int_{|x|>R}|x|^{2\delta-2}|v|^{2}\notag\\
& + \delta t\int_{|x|>R}|x|^{\delta}\Big|\D^{r} v \Big|^{2} \leq \frac{\delta t(d+\delta-2)}{2}\left(\frac{3d-5}{2}+\delta \right)\int_{|x|>R}|x|^{\delta-2}|v|^{2}\notag\\
& + \eta(R)\int_{|x|>R}(|v|^{2}+|\nabla_{A}v|^{2})+\frac{1}{2}\int_{S _{R}}\lambda|x||v|^{2}\notag\\
& +\left(\frac{d-1}{4} + \frac{R}{2} +\eta(R) + \frac{\delta^{2}t^{2}R^{2\delta-1}}{8} \right)\int_{S_{R}}(|v|^2+|\D v|^{2}).\notag
\end{align}
Consequently, combining the right-hand side of the above inequality with the left-hand side, for $R$ large enough and for any $t\geq 1$, $0<\delta<2/3$, $\lambda \geq \lambda_{0}$, it follows that
\begin{equation}\notag
\int_{|x| \geq R} |v|^{2}  \leq C_{\delta}\left(t^{2}+R(1+\lambda)\right) \int_{S_{R}} (|v|^{2}+|\D v|^{2}),
\end{equation}
which implies
\begin{equation}\notag
\int_{|x|>2R} |u|^{2}  \leq C_{\delta}e^{-tR^{\delta}}\left(1+\lambda+\frac{t^{2}}{R}\right),
\end{equation}
being $C_{\delta}$ independent of $t$. Thus letting $t \to \infty$, we obtain that $u=0$ almost everywhere in $\{|x|>2R\}$. The unique continuation property (\cite{R}) implies then $u=0$ almost everywhere in $\Rd$. 

Finally assume that the Sommerfeld radiation condition (\ref{radiacion0}) holds. Moreover, observe that solutions of (\ref{homogeneo}) satisfy (just multiply by $\bar{u}$ and integrate over a ball of radius R), 
$$
\Im \int_{|x|=R} \frac{x}{|x|}\cdot \D u\bar{u} =0.
$$
Hence, we have 
$$\int_{|x|=R} (|\D u|^{2} + \lambda|u|^{2}) d\sigma(x) = \int_{|x|=R} \left|\D u - i\lambda^{1/2}\frac{x}{|x|}u\right|^{2}d\sigma(x) ,$$
\noindent
which together with (\ref{radiacion0}) establishes (\ref{1.130}). The proof of the theorem is complete.

\section{Appendix}

Our proofs combine three integral identities that are obtained by the standard technique of Morawetz multipliers, using integration by parts (see \cite{F}, Lemma 2.1. and \cite{PV1}, Lemma 2.1.). We remark that the idea of integrating by parts with the covariant form $\D$ is to use the Leibnitz formula
\begin{equation}\label{Leibnitz}
\D(f g) = (\D f)g + f(\nabla g),
\end{equation}
putting all the dissorted derivatives on the solution and the straight derivatives on the multiplier.

In order to carry out the integration by parts argument below, we need some regularity in the solution $u$. In general, it is enough to know that $u\in H^{1}_{A}(\Rd)$. Moreover, since we are including singularities in our potentials, it is necessary to put some restrictions on them to check that the contributions of these terms make sense. To this end, it would be enough to check that
\begin{align}\notag
& \int (\partial_{r}V_{1}) |u|^{2} + \int (V_{1} + V_{2}) |u|^{2} +  \int |x|^{2}|B_{\tau}|^{2}|u|^{2} < \infty,
\end{align}
which is true for our potentials by the magnetic Hardy inequality
\begin{equation}\label{A.1}
\int \frac{|u|^{2}}{|x|^{2}} \leq \frac{4}{(d-2)^{2}} \int |\D u|^{2},
\end{equation}
that holds for any $u\in H^{1}_{A}(\Rd)$ with $d\geq 3$.

Now we are ready to state the key equalities.

\begin{lem}\label{appendix1} Let $\varphi: \Rd \to \R$ be regular enough. Then, the solution $u\in H^{1}_{A}(\Rd)$ of the Helmholtz equation (\ref{2.1}) satisfies
\begin{align}\label{(4.11)}
& \int  \varphi\lambda |u|^{2} - \int \varphi |\D u|^{2}  + \int \varphi (V_{1} + V_{2}) |u|^{2}- \Re \int \nabla \varphi \cdot \D u \bar{u}=  \Re\int  \varphi f\bar{u},
\end{align}
\begin{equation}
\varepsilon \int \varphi |u|^{2} - \Im \int \nabla \varphi\cdot \D u\bar{u} = \Im\int \varphi f\bar{u}.\label{(4.2)}
\end{equation}
\end{lem}

\begin{remark}
Note that if we take $\varphi =1$, then we obtain the following a priori estimates 
\begin{align}\label{a+++0}
\varepsilon \int  |u|^{2} \leq \int |f| |u|
\end{align}
\begin{align}\label{b+++0}
\int |\D u|^{2} & \leq \int (\lambda + V_{1} + V_{2})|u|^{2}  + \int |f|||u|,
\end{align}
that have been very useful throughout the paper.
\end{remark}

\begin{lem}\label{appendix2}
Let $\psi: \mathbb{R}^d \longmapsto \mathbb{R}$ be radial, regular enough. Then, any solution $u\in H^{1}_{A}(\Rd)$ of the equation (\ref{2.1}) satisfies

\begin{align}\label{(4.3)}
&\int \D u\cdot D^{2} \psi \cdot \overline{\D u} +\Re \frac{1}{2}\int \nabla(\Delta \psi)\cdot \D u \bar{u}+ \varepsilon \Im\int \nabla \psi\cdot \overline{\D u}u \\
& - \Im  \int \psi' B_{\tau} \cdot \D u\bar{u} -\frac{1}{2}\int \Delta\psi V_{2} |u|^{2} -\Re\int V_{2}\nabla\psi\cdot \D u\bar{u} \notag\\
& + \frac{1}{2} \int \psi' \partial_{r}V_{1} |u|^{2} = -\Re \int f\nabla \psi \cdot \overline{\D u} -\frac{1}{2}\Re\int f \Delta\psi\bar{u},\notag
\end{align}
where $D^{2} \psi$ denotes the Hessian of $ \psi$.
\end{lem}

\begin{remark}\label{Btauremark}
The integration by parts gives very precise information about the relevant quantities related to the electromagnetic field. It is of a particular interest the part concerning the magnetic potential $A$. Note that in the above identities only appear the tangential component of the magnetic field, i.e., the quantity $B_{\tau}$. 
\end{remark}

\end{document}